\documentclass[hidelinks,onefignum,onetabnum]{siamart250211}

\usepackage{ulem}

\usepackage{a4wide}
\usepackage{graphicx}
\usepackage{xcolor}
\usepackage{amssymb,amsmath,amsbsy}
\usepackage{algorithm}
\usepackage{algpseudocode}

\usepackage{tikz}

\usepackage{pgfplots,pgfplotstable}
\pgfplotsset{
        table/search path={../},
    }

\graphicspath{{../}}
\usepgfplotslibrary{groupplots}

\newsiamremark{remark}{Remark}

\newcommand*{\mesh}{{\ensuremath{\mathcal{T}_h}}}
\newcommand*{\elem}{\ensuremath{T}}
\newcommand*{\face}{\ensuremath{F}}
\newcommand*{\facets}{{\ensuremath{\mathcal{F}_h}}}

\newcommand*{\interpol}{\ensuremath{\mathcal I}}

\newcommand{\mean}[1]{\{\!\!\{ #1 \} \!\!\}}
\newcommand{\jump}[1]{[\![ #1 ] \!]}
\newcommand{\eps}{\varepsilon}

\newcommand*{\IR}{\mathbb R}
\newcommand*{\IP}{\mathbb P}

\newcommand{\uu}{\underline{u}}

\newcommand*{\dx}{\ensuremath{\; \textup d\boldsymbol{x}}}
\newcommand*{\ds}{\ensuremath{\; \textup d\sigma}}
\newcommand*{\normal}{\ensuremath{\boldsymbol{n}}}
\newcommand*{\uD}{\ensuremath{u_{\mathrm{D}}^{}}}
\newcommand*{\bx}{\boldsymbol{x}}

\usepackage{booktabs}
\usepackage{siunitx}
\usepackage{tcolorbox}
\tcbuselibrary{skins}

\newtcolorbox{outerbox}{colback=black!5!white, enhanced, frame hidden, borderline west = {1pt}{-1pt}{black,dashed}}
\newtcolorbox{innerbox}[1][]{colback=white, enhanced, frame hidden,width=0.95\textwidth, before = \hspace{0.3cm},borderline west = {1pt}{-1pt}{black,dashed}}

\usepackage{todonotes}

\pgfplotscreateplotcyclelist{philipcolors}{%
violet,very thick, mark size=3pt, mark=*, mark options={scale=0.5}\\%
teal,  very thick, mark size=3pt, mark=*, mark options={scale=0.5}\\%
orange,very thick, mark size=3pt, mark=*, mark options={scale=0.5}\\%
cyan,  very thick, mark size=3pt, mark=*, mark options={scale=0.5}\\
purple,very thick, mark size=3pt, mark=*, mark options={scale=0.5}\\
pink,  very thick, mark size=3pt, mark=*, mark options={scale=0.5}\\
violet,very thick\\%
teal,  very thick\\%
orange,very thick\\%
cyan,  very thick\\
purple,very thick\\
pink,  very thick\\
}

\definecolor{lateksii_color}{RGB}{155,0,119}

\usepackage{hyperref}

\headers{Bound-preserving and conservative EDG methods}{G.R.Barrenechea, P.L. Lederer, and A. Rupp}

\title{A bound-preserving and conservative enriched Galerkin method for elliptic problems\thanks{G.R. Barrenechea has been supported by the Leverhulme Trust through the Research Project Grant  No. RPG-2021-238. 
A.\ Rupp has been supported by the Research Council of Finland's decision numbers 350101, 354489, 359633, 358944.}}

\author{Gabriel R. Barrenechea\thanks{Department of Mathematics and Statistics, University of Strathclyde, Glasgow, Scotland 
  (\email{gabriel.barrenechea@strath.ac.uk}).}
  \and Philip L.\ Lederer \thanks{
    Department of Mathematics, Faculty
of Mathematics, Informatics and Natural Sciences, University of Hamburg,
Hamburg, Germany (philip.lederer@uni-hamburg.de)}
\and Andreas Rupp
\thanks{
    Department of Mathematics, Saarland University, 66123 Saarbrücken, Germany (andreas.rupp@uni-saarland.de)}
  }

\begin{document}
\maketitle
\begin{abstract}
 \noindent We propose a locally conservative enriched Galerkin scheme that preserves the physical bounds for an elliptic problem. To this end, we use a substantial over-penalization of the discrete solution's jumps to obtain optimal convergence. To avoid the ill-conditioning issues that arise in over-penalized schemes, we introduce an involved splitting approach that separates the system of equations for the discontinuous solution part from the system of equations for the continuous solution part, yielding well-behaved subproblems. We prove the existence of discrete solutions and optimal error estimates, which are validated numerically.
\end{abstract}

\begin{keywords}
 Maximum principle, bound preservation, local conservation, enriched Galerkin method
\end{keywords}

\begin{MSCcodes}
% 65N30: Finite elements, Rayleigh-Ritz and Galerkin methods, finite methods
% 65N15: Error bounds for boundary value problems involving PDEs
% 65N12: Stability and convergence of numerical methods for boundary value problems involving PDEs
% 35J25: Boundary value problems for second-order elliptic equations
% 35B50 – Maximum principles
65N30, 65N15, 65N12, 35J25, 35B50
\end{MSCcodes}

% --------------------------------------------------------------------------------------------------
\section{Introduction}
% --------------------------------------------------------------------------------------------------
% Structure-preserving 
Physically-consistent numerical methods have become a key focus in the numerical analysis of partial differential equations (PDEs). These methods aim to preserve specific properties of the exact solutions, such as local conservation laws, entropy conditions, maximum principles, or positivity constraints. Such properties are critical for ensuring that numerical simulations of complex phenomena remain physically accurate and stable. 

One particular type of physically consistent method that has received much attention are those that stay within an acceptable set of values, that is, where the discrete solution does not attain values that the continuous one would not. This is usually carried out by developing discretizations that satisfy the so-called \textit{Discrete Maximum Principle (DMP)}, which essentially states that the discrete solution remain within the bounds of the continuous one and does not have any local extrema that are not present in the continuous solution. This is not a new idea, and in fact in the finite element literature can be traced back to the work of Ciarlet \& Raviart \cite{CR73}. Since then, many different approaches have been proposed, and herein we refert to \cite{Kik77,MH85,BE05,BKK08,BBK17} for some representative methods that respect the DMP in the elliptic case, and to \cite{BJK25} for a very recent review on the topic.

Now, the DMP is sometimes difficult to achieve, as in fact its validity is usually linked to mesh restrictions, and in some situations might even be too strict a requirement. In fact, sometimes a discretization that is just \textit{bound-preserving}, that is, once whose solution remains within the \textit{invariant domain} (i.e., the set of values that the continuous solution takes), is sufficient to guarantee stability of a simulation. In addition, methods that focus on bound preservation are less restrictive and, thus, potentially more efficient for various applications. The importance of bound-preserving methods is underscored in scenarios where, e.g., positivity is required to maintain the validity of a physical model. For instance, negative solutions would be non-physical in reaction-diffusion systems or turbulence models, and phase-field models introduce variables that need to belong to the interval $[-1,1]$. A stable numerical simulation can be built only by ensuring this requirement, without the need to fulfill the stronger restrictions demanded by the DMP.

In the recent work \cite{BarrenecheaGPV23}, the idea of bound-preservation at the nodes was presented for an elliptic reaction-diffusion problem. The core idea involves choosing a baseline discretization of the PDE at hand and combining it with a projection operator that maps finite element functions to their positive (well-behaved) counterparts, ensuring that numerical solutions remain within the required bounds. As anticipated by Godunov's barrier theorem, the method is nonlinear. For symmetric linear problems, the solution is sought as an orthogonal projection of the exact solution onto the convex set of finite element functions with positive nodal values. This work has been recently extended to convection-dominated problems in \cite{AmiriBP23}, using continuous finite element spaces, and to time-dependent problems using discontinuous Galerkin in \cite{BPT25}. A common feature of all the above works is the fact that they can be proven to be equivalent to solving a variational inequality over the constrained convex set (in fact, the method presented in \cite{BPT25} is written directly in such a way, without the link to a stabilized method). This fact prevents the method from being conservative (either locally or globally), a very desirable property in, for example, time-dependent transport problems.

Based on the discussion in the above paragraph, in this work \textbf{we present a method that preserves the bounds of the continuous solution while at the same time is locally (and globally) conservative.} The baseline discretization is the Enriched Galerkin (EG) method, in which the constraints are encoded via Lipschitz-continuous projections. The definition of the projection is one of the main differences with the previous works \cite{BarrenecheaGPV23,AmiriBP23,BPT25}. In fact, while in those papers the whole finite element function is restricted to belong to the acceptable interval, in the present work we split the finite element solution into its linear and constant parts. Then, the piecewise constant part remains untouched, and it is used to build the projection that will take the piecewise linear into  a suitable interval. As such, the stabilization needed to remove the kernel induced by the projection vanishes in the piecewise constant functions, which guarantees local and global conservation. The method is proven to achieve near-best approximation properties in suitable error norms. In addition, this approach enables the use of simple solvers, such as Richardson-like iterations, thereby avoiding more complex constrained optimization techniques commonly employed in the existing literature \cite{EHS09,Bochev20,MN16}. Moreover, to the authors' knowledge, this is the first work that combines bound-preservation and local conservation with a rigorous error analysis.

The enriched Galerkin (EG) methods use the discontinuous Galerkin (DG) bilinear form and a combination of overall continuous and broken polynomial spaces to inherit the polynomial best approximation property from continuous finite elements and other beneficial properties, such as enhanced stability and local mass conservation from discontinuous Galerkin methods, while maintaining fewer degrees of freedom than DG. The first enriched Galerkin method was proposed by Becker et al.\ \cite{BeckerBHL03}, but these methods only gained popularity after reinvention by Lee et al.\ \cite{LeeLW16}. Since then, the EG method has been continuously generalized and applied to many problems. Its main strengths stem from its flexibility concerning the polynomial spaces that allow, for example, higher-order or spatially adaptive enrichments \cite{RuppHA20,RuppL20}. This flexibility allows us to derive, for instance, locking-free and physics-preserving EG methods for poroelasticity \cite{LeeY23,YiL24} or efficient EG-based solvers for the Stokes problem \cite{YiHLA22}. However, this flexibility obstructs bound-preserving approximations even in the lowest-order cases. This issue has been addressed in \cite{KuzminHR20,KuzminH23,KuzminLY24} for hyperbolic problems and will be addressed for elliptic problems in the current work.

One feature of our work is the fact that, due to the particular choice of how the bounds are imposed in the discrete solution, the jump terms in the EG formulation need to be over-penalized. This is not a new concept; in fact, several works have explored the over-penalization of DG methods, such as the WOPSIP method \cite{BOS08}. Now, it is a well-known fact that over-penalized methods lead to very ill-conditioned linear systems of equations. Consequently, in \cite{BOS08}, a preconditioner has also been introduced to address this issue. To completely circumvent the ill-conditioning, this work exploits the fixed-point strategy used to prove the existence of solutions to propose a decoupled scheme that completely avoids solving ill-conditioned linear systems, thereby eliminating the need to propose appropriate preconditioners. This insight seems to be new even for linear EG methods. 

Several numerical experiments demonstrate the effectiveness of the proposed method compared to standard finite element approaches. These experiments show that the method preserves critical structural properties of the exact solution while maintaining computational efficiency. Similar methods that focus on bound preservation, such as truncation-based approaches \cite{Lu2013TheCM,Kreuzer,YYZ22}, have been previously explored. However, the method presented here offers a fresh perspective on the formulation and numerical solution of bound-preserving problems.

The remainder of this contribution is structured as follows: Section \ref{SEC:model_prob} delves into the model problem and the (not bound-preserving) over-penalized enriched Galerkin discretization. In contrast, Section \ref{sec:postpreserv} turns this discretization into a bound-preserving one. Section \ref{SEC:existence} is devoted to proving the existence of discrete and positive solutions, while Section \ref{SEC:conbergence} discusses their convergence properties to the analytic solution. Finally, Section \ref{SEC:numerics} demonstrates how our method can be implemented and analyzes the condition of the linear system of equations that needs to be solved internally before a short conclusion wraps up this work.

% --------------------------------------------------------------------------------------------------
\section{Model problem and baseline discretization}\label{SEC:model_prob}
% --------------------------------------------------------------------------------------------------
% 
We will adopt standard notation on Sobolev and Lebesgue spaces, aligned with, e.g., \cite{EG21-I}. For $D\subseteq\mathbb{R}^{d}$, $d\le 3$, we denote by $\|\cdot\|_{0,p,D} $ the $L^{p}(D)$-norm; when $p=2$ the subscript $p$
will be omitted and we only write $\|\cdot\|_{0,D} $. In addition, for
$s\geq 0$, $p\in [1,\infty]$, we denote by $\| \cdot \|_{s,p,D}$ ($|
\cdot |_{s,p,D}$) the norm (seminorm) in $W^{s,p}(D)$; when $p=2$, we
define $H^{s}(D)=W^{s,2}(D)$, and again omit the subscript $p$ and only write $\|\cdot \|_{s,D}$
($| \cdot |_{s,D}$).  The following space will also be used repeatedly within the text
\begin{align}
	H^{1}_{0}(D)=\left\{ v \in H^1(D) : v = 0 \; {\rm on} \; \partial D \right\}. \label{space}
\end{align}

Let $\Omega \subset \IR^d, d=2,3$ be a 
polyhedral, bounded, Lipschitz domain, $f \in L^2(\Omega)$,
$\uD \in C^0(\partial \Omega)$,
and let $\epsilon, \mu \in \mathbb{R}^+$. We consider the following elliptic reaction-diffusion equation: Find  $u \in H^1(\Omega)$ such that $u = \uD$ on $\partial \Omega$, and
\begin{equation}\label{EQ:cont_prob}
 \int_\Omega \epsilon \nabla u \cdot \nabla v \dx + \int_\Omega \mu u v \dx = \int_\Omega f v \dx,
\end{equation}
holds for all $v \in H^1_0(\Omega)$. Using Lax-Milgram's Lemma (see, e.g., (see, e.g., \cite[Lem.~25.2]{EG21-II}), this variational problem can be shown to have a unique solution. In addition, using the comparison principle (see \cite[Cor.~4.4]{RenardyR04}), then $u$ satisfies
\begin{equation}\label{Eq:comparison}
\|u\|_{0,\infty,\Omega}\le \tilde{U}:=\max\left\{\frac{\|f\|_{0,\infty,\Omega}^{}}{\mu}, \|u_D\|_{0,\infty,\partial\Omega}^{}\right\}\,.
\end{equation}
In addition, if $f\ge 0$ and $u_D\ge0$ we furthermore have that $u\ge 0$. That is, equation \eqref{EQ:cont_prob} respects (physically meaningful) bounds for its solution.

We now describe the baseline discretization of \eqref{EQ:cont_prob}. The choice made in this work is the EG method, proposed in \cite{BeckerBHL03,LeeLW16}. For this, we consider a shape-regular family of conforming simplicial triangulations $(\mesh)_{h > 0}$. Only to avoid technical diversions and simplify the notations and some of the proofs, we will assume that the family of triangulations is quasi-uniform. For a triangulation $\mesh$, the set of its facets is denoted by $\facets$. For any $\elem \in \mesh$, $h_\elem = \operatorname{diam} \elem$, while $h_\face = \operatorname{diam}\face$ for all $\face \in \facets$, $h = \max_{\elem \in \mesh} h_\elem$, and $h_{\min} = \min_{\elem \in \mesh} h_\elem$. In addition, we denote by $\bx_1,\ldots,\bx_n$ the internal nodes of $\mesh$ and for $i\in\{1,\ldots,n\}$ and $T\in\mesh$ we define the neighborhoods
\begin{equation*}
    \omega_i=\{T\in\mesh: \bx_i\in T\}\quad\textrm{and}\quad\omega_T=\{T'\in\mesh: T\cap T'\not=\emptyset\}\,.
\end{equation*}

The finite element space in the EG method is given by $V_h=V_h^1\oplus V_h^0$, where 
\begin{align*}
\tilde{V}_h^1 &= \{ v \in C(\overline{\Omega}) \colon v|_\elem \in \IP_1(\elem) \; \forall \elem \in \mesh \}, \\
 V_h^1 &= \tilde{V}_h^1\cap  H^1_0(\Omega)\,, \\
 V_h^0 &= \{ v \in L^2(\Omega) \colon v|_\elem \in \IP_0(\elem) \; \forall \elem \in \mesh \}.
\end{align*}
Since $V_h \not \subset H^1(\Omega)$, the EG method is not conforming. Still, $V_h$ is a subset of the broken (linear) polynomial space used in DG methods. Thus, we can use a variant of the DG bilinear and linear forms to define the EG approximate $u_h \in V_h$. To this end, let us define the jump and average operators. If $\face \in \facets$ connects two elements $\elem^+, \elem^- \in \mesh$, i.e., 
$\face = \elem^+ \cap \elem^-$,the average and jump of $v_h \in V_h$ are defined by
\begin{equation*}
 \mean{ v_h } = \tfrac12 v_h|_{\elem^+} + \tfrac12 v_h|_{\elem^-} \qquad \text{ and } \qquad \jump{v_h} = v_h|_{\elem^+} \normal^+ + v_h|_{\elem^-} \normal^-,
\end{equation*}
where $\normal^\pm$ denotes the outward pointing unit normal of element $\elem ^\pm$. Note that the jump turns a scalar function into a vector-valued quantity. If $\face \subset \partial \Omega \cap \elem$ for some $\elem \in \mesh$, we use
\begin{equation*}
 \mean{ v_h } = v_h \qquad \text{ and } \qquad \jump{v_h} = v_h \normal,
\end{equation*}
where $\normal$ is the outward pointing unit normal to $\partial\Omega$. Moreover, we need $L^2$-type scalar products in the mesh's bulk and skeleton, i.e., concerning $\mesh$ and $\facets$. We define
\begin{equation*}
 (w_h, v_h) = (w_h, v_h)_\mesh = \sum_{\elem \in \mesh} \int_\elem w_h v_h \dx \quad \text{ and } \quad \langle \eta_h, \sigma_h \rangle = \langle \eta_h, \sigma_h \rangle_\facets = \sum_{\face \in \facets} \int_\face \eta_h \cdot \sigma_h \ds,
\end{equation*} 
for $w_h,v_h \in V_h$ and face-wise defined vectors $\eta_h, \sigma_h$. Importantly, $(\cdot, \cdot)_\mesh$ trivially extends to expressions like $(\nabla w_h, \nabla v_h)_\mesh$ if we reinterpret the product in its definition. In addition, we define the norms
\begin{equation}\label{Def:Disc-Norms}
 \| v_h \|^2_{\mesh} = (v_h, v_h)_\mesh \qquad \text{ and } \qquad 
 \| \eta_h \|^2_{\facets} = \sum_{\face \in \facets} \frac{\epsilon+\mu h_\face^2}{h_\face} \int_\face \eta_h \cdot \eta_h \ds.
\end{equation}

Using the above notation, the baseline EG method used in this work is defined as: Find $u_h\in V_h$ such that
\begin{subequations}\label{EQ:eg_discretization}
\begin{equation}
 a_h(u_h, v_h) = b_h(v_h) \qquad \text{ for all } v_h \in V_h,
\end{equation}
where $a_h$ denotes the symmetric interior penalty DG bilinear form and $b_h$ incorporates the right-hand side $f$ and the boundary data, i.e.,
\begin{align}
 a_h(w_h, v_h) &= (\epsilon \nabla w_h, \nabla v_h) + (\mu w_h, v_h)
 - \langle \mean{\epsilon \nabla w_h}, \jump{v_h} \rangle - \langle \mean{\epsilon \nabla v_h}, \jump{w_h} \rangle\nonumber\\
 &\qquad + \left\langle \gamma \frac{\epsilon+\mu h_\face^2}{h_\face^{\beta}} \jump{w_h},\jump{v_h} \right\rangle,\\
 b(v_h) &= (f,v_h)- a_h(\tilde{u}_{D,h}^{}, v_h).
 % (f,v_h)- a_h(I_h(\tilde{u}_\textup D), v_h).
\end{align}
\end{subequations}
Above, $\tilde{u}_{D,h}^{}\in\tilde{V}_h^1$ is an extension of the Lagrange interpolant of the boundary datum inside the domain, $\beta\in\mathbb{N}$, $\gamma=\gamma_0^{}L_\Omega^{\beta-1}$, where $\gamma_0^{}>0$ is a non-dimensional penalty parameter, and $L_\Omega^{}$ is a characteristic length of $\Omega$ (for example, it can be taken as the diameter of $\Omega$).

Strictly speaking, \eqref{EQ:eg_discretization} is not the standard EG formulation as our EG version enforces the element-wise constant component \( u^0_h \) to approximate a homogeneous Dirichlet boundary condition (via Nitsche's method, as enforced in DG schemes), while the overall continuous solution component \( u_h^1 \) is set to zero strongly on the boundary (since  \( V^1_h\subset H^1_0(\Omega)\)). Eventually, we have that $u_h + \tilde{u}_{D,h}^{}\approx u$. For a standard EG discretization, the choice of extension $\tilde{u}_{D,h}^{}$ of the boundary datum is irrelevant, but due to its nonlinearity,  for the method presented below it is simpler to follow the approach given in \cite[Sec.~4.1]{AmiriBP23} and take $\tilde{u}_{D,h}^{}$ to vanish in all the interior nodes. This approach is also useful as we want to enforce that \( u_h + \tilde{u}_{D,h}^{} \) preserves the prescribed bounds at all (interior) nodes, and this is simplified by the present choice for  \( \tilde{u}_{D,h}^{}  \).

The aforementioned treatment of the boundary conditions allows the solution component \( u^0_h \) to converge to zero rapidly, which will be a key aspect of our analysis. Thereby, \( u^0_h \) is only needed to establish the local mass conservation, but does not contribute to the convergence of the method.

\begin{remark}[Properties of the EG method]\label{rem::prop_EG}
 \begin{enumerate}
 \item In the standard EG (and DG) literature, the choice is $\beta = 1$. Nevertheless, it will become clear in the analysis presented below that we will need larger values for $\beta$. So, for the moment, we only state that $\beta\ge 1$. 
 \item Observing that $V_h$ is a subspace of the standard DG space, and that $a_h(\cdot,\cdot)$ is an elliptic bilinear form in the DG space, and so it is elliptic in $V_h$ as well. More precisely, using standard arguments (see, e.g., \cite{PietroE12}), it can be  shown that if $\gamma$ is large enough, there exists $c_0>0$, independent of $h$ and any physical constant, such that
  \begin{equation}\label{Eq:a-elliptic}
 a_h(v_h,v_h)\ge c_0\Big(\|\epsilon^{1/2}\nabla v_h\|^2_{0,\Omega}+\| \mu^{1/2}v_h\|_{0,\Omega}^2+\sum_{F\in\facets}\gamma\frac{\epsilon+\mu h_\face^2}{h_F^\beta}\|\jump{v_h}\|_{0,F}^2\Big)=: c_0^{}\|v_h\|_{a_h}^2\,.
  \end{equation}
  Then, the well-posedness of \eqref{EQ:eg_discretization} follows.
  \item\label{IT:eg_conserv} The  method \eqref{EQ:eg_discretization} is locally mass conservative, as demonstrated by testing with $v_h = 1$ in one element $\elem \in \mesh$ and $v_h = 0$ in all other elements $\elem \in \mesh$. 
  In this case, the bilinear form \( a_h( w_h , 1|_\elem) \) is restricted to a single element \( \elem \in \mesh \) and reads
  \[
   a_h(w_h, 1|_\elem) = \int_\elem \mu w_h \dx + \sum_{\face \in \facets} \int_{\face \cap \partial \elem} \left[ \gamma \frac{\epsilon+\mu h_\face^2}{h_\face^{\beta}} \jump{w_h} - \mean{\epsilon \nabla w_h} \right] \cdot \normal_\elem \ds.
  \]
  The term in brackets is commonly defined as numerical flux, and it is uniquely defined on inter-element (and boundary) faces. Thus, multiplying it with the element normal \( \normal_\elem\) makes the face-wise integral anti-symmetric, which is the characterizing property for local mass conservation.
  \item  As it was mentioned earlier, we allow the possibility of taking $\beta\ge 1$.
  The error analysis requires taking $\beta=4$, which is the value considered in the numerical experiments (although for completeness, results using smaller values will also be presented). If we were to implement the EG method with this negative power, it would significantly affect the condition number of the associated matrix. Motivated by this fact, within the context of our bound-preserving method presented below, we have also introduced an iterative scheme that completely bypasses this issue.
  \item Strictly speaking, a standard symmetric interior penalty DG bilinear form would motivate the definition of the penalty (with a varying $\beta$ scaling) by 
  $
      \langle \gamma \frac{\epsilon}{h_\face^{\beta}} \jump{w_h},\jump{v_h} \rangle.
      $
  Thus, compared to our definition, it would not include the terms involving $\mu$. While this choice will aid in the analysis, it also has a positive impact on stability with respect to the choice of $\gamma$, as discussed in the numerical section.
 \end{enumerate}
\end{remark}

% --------------------------------------------------------------------------------------------------
\section{Positivity-preserving finite elements} \label{sec:postpreserv}
% --------------------------------------------------------------------------------------------------
% 
As discussed, we require our EG approximate $u_h$ to remain within some invariant domain $\mathcal G = [a,b]$, but the EG method generally does not guarantee that  $u_h(\bx)\in \mathcal G$ for all $\bx\in \Omega$. 
Thus, we need to correct it such that our corrected approximate $u_h^+ \in V_h$
\begin{enumerate}
 \item\label{IT:bound_preserving} takes values in $\mathcal G$ only,
 \item\label{IT:mass_conserving} is locally mass conservative and optimally convergent.
\end{enumerate}

To construct $u^+_h$ let us define a projection operator that attempts to correct a function $v_h \in V_h$ such that it only takes values in $\mathcal G$.  The construction of this operator will be based on the split of $V_h$ as piecewise linear and piecewise constant parts. So, every $v_h\in V_h$ is split as 
\begin{equation}\label{EQ:decomp}
 v_h = v^1_h + v^0_h \qquad \text{ where } v_h \in V_h, \; v^1_h \in V^1_h, \; v^0_h \in V^0_h\,.
\end{equation}
Next, we denote by %\bx_1,\ldots, \bx_N$ the internal nodes of $\mesh$, and by 
$\phi_1,\ldots,\phi_N$ are the standard basis functions (the ``hat'' functions) spanning the space $V_h^1$. Then, given $w_h=w_h^1+w_h^0\in V_h$ we define the truncation 
\begin{gather}
 P^{w_h}_i (v_h) = \max\left[ a - \underline w_{hi}, \; \min\left( v^1_h(\bx_i), b - \overline w_{hi} \right) \right], \qquad \text{ where} \label{EQ:def_p}\\
 \underline w_{hi} = \min\{ w^0_h(\bx)\colon \bx \in \omega_i \} \qquad \text{ and } \qquad \overline w_{hi} = \max\{ w^0_h(\bx)\colon \bx \in \omega_i \}.\notag
\end{gather}
Using this truncation, we  define the (nonlinear) mapping
\begin{equation*}
 P^{w_h}_h\colon V_h \ni v_h \mapsto P^{w_h}_h(v_h) = \sum_{i = 1}^N P^{w_h}_i (v_h)\phi_i + w^0_h \in V_h.
\end{equation*}

Additionally, we introduce the operator $Q^{w_h}_h (v_h) = v^1_h - [P^{w_h}_h (v_h)]^1$, and emphasize that the superscripts $0$ and $1$ always refer to the decomposition according to \eqref{EQ:decomp}. Finally, we abbreviate
\begin{equation*}
 v^+_h = v^{1+}_h + v^{0}_h = P^{v_h}_h (v_h) \in V_h \qquad \text{ and } \qquad v^-_h = v^{1-}_h = Q^{v_h}_h (v_h) \in V^1_h \subset V_h,
\end{equation*}
where $v^{1+}_h = \sum_{i = 1}^N P^{v_h}_i (v_h)\phi_i$. The effect of these operators is illustrated in Figure \ref{FIG:projection} for the one-dimensional case.

\begin{figure}[ht]\centering
\newcommand{\drawmesh}{
   \draw[fill=black] (0,0) circle [radius=0.3mm];
   \draw[fill=black] (1,0) circle [radius=0.3mm];
   \draw[fill=black] (2,0) circle [radius=0.3mm];
   \draw[fill=black] (3,0) circle [radius=0.3mm];
   \node[] at (0.5,-0.3) {$T_1$};
   \node[] at (1.5,-0.3) {$T_2$};
   \node[] at (2.5,-0.3) {$T_3$};
   \draw[dashed, black!70!white] (0,0) -- (3,0);
   \draw[dashed] (0,1) -- (3,1);
}
%constants
\def\czero{0.1}
\def\cone{-0.15}
\def\ctwo{0.25}

%linear
\def\lzero{0.4}
\def\lone{1.2}
\def\ltwo{0.0}
\def\lthree{0.6}

%input: constant left, constant right, linear
\newcommand{\project}[3]{%
   % {max(1-min(#1,#2), min(#3,-max(#1,#2)))}%
   {max(0-min(#1,#2), min(#3,1-max(#1,#2)))}%
}
\newcommand{\plus}[4]{%
   {max(0-min(#1,#2), min(#3,1-max(#1,#2)))+#4}%
}
\newcommand{\minus}[3]{%
   {#3-max(0-min(#1,#2), min(#3,1-max(#1,#2)))}%
}

\begin{tikzpicture}[scale=2] 
 \drawmesh
 
 % vh0
 \draw[purple, line width=0.3mm] (0,\czero) -- (1,\czero);
 \draw[purple, line width=0.3mm] (1,\cone) -- (2,\cone);
 \draw[purple, line width=0.3mm] (2,\ctwo) -- (3,\ctwo);

 %vh1
 \draw[orange, line width=0.3mm] (0,\lzero) -- (1,\lone) -- (2,\ltwo) -- (3,\lthree);
 
% vh
   \draw[cyan, line width=0.3mm] (0,\czero+\lzero) -- (1,\czero+\lone);
   \draw[cyan, line width=0.3mm] (1,\cone+\lone) -- (2,\cone+\ltwo);
   \draw[cyan, line width=0.3mm] (2,\ctwo+\ltwo) -- (3,\ctwo+\lthree);

 \begin{scope}[xshift=4cm]
   \drawmesh

    % vh^{1+}
   \draw[orange, line width=0.3mm] 
        (0,\project{\czero}{\czero}{\lzero}) -- 
        (1,\project{\czero}{\cone}{\lone}) -- 
        (2,\project{\cone}{\ctwo}{\ltwo}) -- 
        (3,\project{\ctwo}{\ctwo}{\lthree});

   % vh^{+}
   \draw[cyan, line width=0.3mm] (0,\plus{\czero}{\czero}{\lzero}{\czero}) -- (1,\plus{\czero}{\cone}{\lone}{\czero});
   \draw[cyan, line width=0.3mm] (1,\plus{\czero}{\cone}{\lone}{\cone}) -- (2,\plus{\cone}{\ctwo}{\ltwo}{\cone});
   \draw[cyan, line width=0.3mm] (2,\plus{\cone}{\ctwo}{\ltwo}{\ctwo}) -- (3,\plus{\ctwo}{\ctwo}{\lthree}{\ctwo});

   % vh^{-}
   \draw[purple, line width=0.3mm] (0,\minus{\czero}{\czero}{\lzero}) -- (1,\minus{\czero}{\cone}{\lone});
   \draw[purple, line width=0.3mm] (1,\minus{\czero}{\cone}{\lone}) -- (2,\minus{\cone}{\ctwo}{\ltwo});
   \draw[purple, line width=0.3mm] (2,\minus{\cone}{\ctwo}{\ltwo}) -- (3,\minus{\ctwo}{\ctwo}{\lthree});

 \end{scope}

 \node at (3.5,0) {$a = 0$};
 \node at (3.5,1) {$b = 1$}; 
\end{tikzpicture}
\caption{Illustration of $P_h^{v_h}$ on three elements $T_1, T_2, T_3$. Left: $v_h$ (cyan) and $v^1_h$ (orange) and $v_h^{0}$ (purple). Right: $v_h^+$ (cyan) and  $v_h^{1+}$ (orange) with $v_h^-$ (purple). In both cases, $a=0$ and $b=1$.}\label{FIG:projection}
\end{figure}

The mapping $P_h^{v_h}$ is built with the aim of guaranteeing that $P_h^{v_h}(\bx)\in\mathcal G$ for all $\bx\in\Omega$. This is not necessarily true for every $v_h\in V_h$. In the next result, we state the condition under which this fact can be achieved, and we will identify criteria that inform the assumptions made in Section \ref{SEC:positivity}.
\begin{lemma}\label{COR:positivity_preservation}
 Let $v_h \in V_h$. If $a - \underline v_{hi} \le b - \overline v_{hi}$ for all node indices $i = 1, \dots, N$, then $v^+_h(\bx) \in \mathcal G$ for all $\bx\in\Omega$. 
\end{lemma}

\subsection{The nonlinear finite element method} 

With the notations introduced in the last section, we introduce the finite element method studied in this work: Find $u_h\in V_h$ such that
\begin{equation}\label{EQ:stab_prob}
 a_h(u^+_h, v_h) + s_h(u^-_h, v_h) = b_h(v_h)\qquad\forall\, v_h\in V_h\,.
\end{equation}
Here, the bilinear form $a_h$ and linear form $b_h$ are defined as in \eqref{EQ:eg_discretization}, while the stabilization term for any $v_h, w_h \in V_h$ reads
\begin{equation}\label{EQ:stabilizer}
 s_h(w_h, v_h) = \alpha \sum_{i=1}^N (\epsilon h_i^{d-2} + \mu h_i^d) w^1_h(\bx_i) v^1_h(\bx_i).
\end{equation}
Here, $\alpha > 0$ is a nondimensional stabilization parameter, and $h_i = \max\{ h_\elem \mid \elem \in \mesh \text{ with } \elem \subset \omega_i \}$. The stabilization term $s_h$ induces the norm $\| \cdot \|_{s_h}$ on $V_h^1$ (and thus a seminorm in $V_h$). 

The finite element method just presented is nonlinear since $u_h^+$ is a nonlinear transformation of $u_h$. So, it requires an appropriate linearization. In addition, as mentioned earlier, the presence of the term $h_F^{-\beta}$ in the penalty can affect the condition number of the linear subproblems tremendously. To address these two issues, we now present a fixed-point iterative algorithm that enables us to both prove the existence of solutions and circumvent the potential ill-conditioning of the scheme. For this, we define the mapping
\begin{subequations}
\begin{equation}\label{EQ:fixed_point}
 \widetilde{T}\colon V^0_h \ni w^0_h \mapsto w^1_h \mapsto \tilde w^0_h \in V^0_h\,,
\end{equation}
by the following algorithm:

\begin{outerbox}%{Outer loop}

\textbf{Step 1}: Compute $w^1_h\in V_h^1$ by solving the nonlinear scheme:
\begin{equation}\label{EQ:T_1}
 a_h([P^{w^0_h}_h w_h]^1, v^1_h) + s_h([Q^{w^0_h}_h w_h]^1, v^1_h) = (f, v^1_h)_\mesh - a_h(w^0_h, v^1_h) \qquad \forall v^1_h \in V^1_h.
\end{equation}

\textbf{Step 2}: Compute $\tilde w^0_h \in V^0_h$ as solution of
\begin{equation}\label{EQ:T_2}
 a_h(\tilde w^0_h, v^0_h) = (f, v^0_h)_\mesh - a_h(w^{1+}_h, v^0_h) - \underbrace{ s_h(w^{1-}_h, v^0_h) }_{ = 0 } \qquad \forall v^0_h \in V^0_h.
\end{equation}
\end{outerbox}
\end{subequations}

In the analysis presented in the next section, we will show that the operator $\widetilde{T}$  is well-defined. In addition, we make the following observations:
\begin{enumerate}
 \item a function \( \tilde w_h = \tilde w_h^1 + \tilde w_h^0 \) comprising the fixed-point \( \tilde w^0_h \) of $\widetilde{T}$ and the solution of \eqref{EQ:T_1} with \( w^0_h \) replaced by \( \tilde w^0_h \) solves \eqref{EQ:stab_prob}, which we can easily see by adding \eqref{EQ:T_1} and \eqref{EQ:T_2}. So, in the next section we will show that $\widetilde{T}$ has a fixed point;
 \item since the problem \eqref{EQ:T_1} is posed over the space $V_h^1$ (which contains only continuous functions), the jump terms vanish. Thus, the condition number of the linear problems needed to solve \eqref{EQ:T_1} is independent of $\beta$;
 \item the left-hand side of \eqref{EQ:T_2} can be simplified to
 \begin{equation*}
(\mu \tilde w_h^0, v_h^0)_\mesh  + \sum_{F\in\facets}\frac{\gamma (\epsilon+\mu h_\face^2)}{h_F^\beta}\int_F \jump{\tilde w_h^0}\cdot \jump{v_h^0}\,.
 \end{equation*}
 Since the mesh is assumed to be shape-regular, then the matrix associated with \eqref{EQ:T_2} can be proven to have a condition number independent of $\beta$. We shall provide more details on this in Section~\ref{SEC:condtioning}.
\end{enumerate}

\begin{remark}[Bound preservation and mass conservation]
 Clearly, \( u^+_h \) takes only values in \( \mathcal G \) achieving the target described in item \ref{IT:bound_preserving}. Moreover, the stabilization \( s_h(u_h^-, v_h) \) vanishes if the \( v_h \in V_h \) is chosen as the characteristic function of one mesh element, i.e., from the space \( V^0_h \). So, the method \eqref{EQ:stab_prob} reduces to the terms only included in \eqref{EQ:eg_discretization} for such a test function, and thus the solution \( u^+_h \) satisfies the local mass conservation criterion of item \ref{IT:mass_conserving} (see also item~\ref{IT:eg_conserv} in Remark~\ref{rem::prop_EG}) with the same numerical inter-element flux as the EG scheme defined in \eqref{EQ:eg_discretization}.
\end{remark}

% --------------------------------------------------------------------------------------------------
\section{Existence of a discrete solution}\label{SEC:existence}
% --------------------------------------------------------------------------------------------------
% 
This section establishes the existence and stability of a solution to \eqref{EQ:stab_prob}, which is made explicit in the following Theorem:

\begin{theorem}\label{TH:existence}
 Let $\beta \ge 2$ and let us suppose that $\alpha, \gamma > 0$ are sufficiently large. 
 Then, there exists a solution $u_h=u_h^1+u_h^0\in V_h$ of \eqref{EQ:stab_prob}. In addition, the piecewise constant part $u_h^0$ satisfies the following a priori bound
 \begin{equation}\label{EQ:stab_result}
\| \jump{ u^0_h } \|^2_\facets 
\le \, C\,\frac{h^{2\beta- 4}}{\gamma^2} \left[ C_P^2\mu^{-1}\| f \|^2_{0,\Omega} + (C_P^2\mu+\epsilon)\max(|a|,|b|)^2 |\Omega| \right]\,,
 \end{equation}
 where $C>0$ depends only on the shape-regularity of the mesh.
\end{theorem}

To show this result, we will exploit the fixed-point iteration defined in \eqref{EQ:fixed_point}. Thus, we need to show that the subproblems \eqref{EQ:T_1} and \eqref{EQ:T_2} are well-posed, i.e., that Step 1 and Step 2 in the construction of \( \tilde T \) are well-defined and continuous operations. We do this in Lemmas \ref{EQ:wellposed_T1_real} and \ref{LEM:t2_wellposed}, respectively. Afterward, we establish the fact that \( \tilde T \) admits at least one fixed-point by applying Brouwer's fixed point theorem in an adequately constructed ball, cf.\ Section \ref{SEC:proof_existence}. To this end, we use that \( \tilde T \) is continuous (which follows from Lemmas \ref{EQ:wellposed_T1_real} and \ref{LEM:t2_wellposed}), and and show that it maps the adequately constructed ball to itself. We start this endeavor by providing auxiliary results.

\subsection{Preliminaries and auxiliary results}
We exploit our fixed-point iteration \eqref{EQ:fixed_point} to prove this result. Thus, we need to prove that \eqref{EQ:T_1} and \eqref{EQ:T_2} induce well-posed problems, which we perform in Lemmas \ref{LEM:T1_wellposed} and \ref{LEM:t2_wellposed}, respectively. To this end, we need several auxiliary results.
\begin{lemma}[Broken Poincar\'e inequality, \cite{Brenner03}]\label{LEM:poincare}
 For every $v^0_h \in V^0_h$ the following broken Poincar\'e inequality holds
 \begin{equation*}
  \| v^0_h \|_{0,\Omega}^{} \le\, C_P^{} \left\{\sum_{\face\in\facets}\frac{1}{h_\face^{}}\|\jump{v_h^0}\|^2_{0,\face}\right\}^{\frac{1}{2}}\,,
 \end{equation*}
 where $C_P^{}>0$ depends only on $\Omega$ and the regularity of the mesh.
\end{lemma}

\begin{lemma}
 There is a constant $C_\textup{equiv}^{}$, which only depends on the mesh regularity, such that
 \begin{equation*}
  \left(\mu\|v_h^1\|_{\mathcal{T}_h}^2+\epsilon \|\nabla v_h^1\|_{\mathcal{T}_h}^2\right)^{1/2}=\| v^1_h \|_{a_h}^{} \le \frac{C_\textup{equiv}^{}}{\sqrt{\alpha}} \| v^1_h \|_{s_h}^{} \qquad \text{ for all } v^1_h \in V^1_h.
 \end{equation*}
\end{lemma}
\begin{proof}
    The proof of a very similar result can be seen in \cite[Eq.~(3.8)]{BarrenecheaGPV23}, but we summarize it here for completeness. First, using the following inverse inequality (see \cite[Lem.~12.1]{EG21-I}): There is a constant $\tilde c>0$, independent of $h$ such that
    \begin{equation}\label{Eq:inverse}
        \tilde c h_\elem^2 | v^{1}_h |^2_{1,\elem} \le \| v^{1}_h \|^2_{0,\elem},
    \end{equation}
    for all $v^{1}_h\in V^1_h$. 
    In addition, we recall the following inequality, whose proof can be found, e.g., in \cite[Prop.~12.5]{EG21-I}:
 \begin{equation}\label{Equiv:nodal-values}
  \| v^{1}_h \|^2_{0,\elem} \le\, C \sum_{\boldsymbol{x}_i\in\elem} h_i^d\, [v^{1}_h (\bx_i)]^2\,,
 \end{equation}
 for all $v^{1}_h\in V^1_h$. Then, using the inverse inequality \eqref{Eq:inverse} followed by \eqref{Equiv:nodal-values} we get to
 \begin{align*}
    \mu\|v_h^1\|_{\mathcal{T}_h}^2+\epsilon \|\nabla v_h^1\|_{\mathcal{T}_h}^2
    &\le C\,\sum_{\elem\in\mathcal{T}_h}(\mu+\epsilon h_\elem^{-2})\|v_h^1\|_{0,\elem}^2\\
    &\le C\,\sum_{\elem\in\mathcal{T}_h}(\mu+\epsilon h_\elem^{-2})\sum_{\boldsymbol{x}_i\in\elem}h_i^d[v^{1}_h (\bx_i)]^2\\
    &\le \frac{C}{\alpha}\,s(v_h^1,v_h^1)\,,
 \end{align*}
 where in the last step we used the mesh regularity. This finishes the proof.
\end{proof}

\begin{lemma}\label{LEM:stab1}
Given $v^0_h \in V^0_h$, there is a constant  $C$  (independent of $h$ and $v^0_h$) such that 
 \begin{equation*}
   \| v^{1+}_h \|^2_{0,\Omega} \le C \left[ \max(|a|,|b|)^2 |\Omega| + \| v^0_h \|^2_{0,\Omega} \right], \quad \text{ for all } v^1_h \in V^1_h.
 \end{equation*}
\end{lemma}
\begin{proof}
 Let $v_h^1\in V_h^1$. The definition of $v^{1+}_h(\bx_i)$ implies that $|v^{1+}_h(\bx_i)| \le \max(|a|,|b|) + \| v^0_h \|_{\infty,\omega_i}^{}$. 
Using this fact,  the inverse inequality \eqref{Eq:inverse}, the scaling \eqref{Equiv:nodal-values}, and the mesh regularity we get to
 \begin{equation*}
   \| v^{1+}_h \|^2_{0,\Omega} \le\, C \sum_{i = 1}^N [\max(|a|,|b|) + \| v^0_h \|_{\infty,\omega_i}^{}]^2 h_i^d \le\, C \left[ \max(|a|,|b|)^2 |\Omega| + \| v^0_h \|^2_{0,\Omega} \right],
 \end{equation*}
 which finishes the proof.
\end{proof}

In addition, we also recall the following trace-inverse inequality (see \cite[Lem.~12.10]{EG21-I}):
\begin{lemma}
 There exists $C>0$, depending only on the mesh regularity such that for any $T\in\mesh$ and a facet $F$ of $T$, and any $v_h^{}\in V_h^{}$ the following holds 
 \begin{equation}\label{Local:trace-inverse}
  \|v_h^{}\|_{0,F}^{}\le C\,h_T^{-1/2}\|v_h^{}\|_{0,T}^{}\,.
 \end{equation}
\end{lemma}

The next two results are instrumental in the proof that the nonlinear problem from Step 1 (cf. \eqref{EQ:T_1}) is well-posed. The first result is an generalization of \cite[Lem.\ 3.1]{BarrenecheaGPV23} to the present (more involved) case, so we present a summarized proof.

\begin{lemma}\label{LEM:generalize_Barranachea}
 For given $w^0_h \in V^0_h$, we have the following relations for all $r^1_h, v^1_h \in V^1_h$
 \begin{equation*}
  s_h([Q^{w^0_h}_h r^1_h]^1 - [Q^{w^0_h}_h v^1_h]^1, [P^{w^0_h}_h r^1_h]^1 - [P^{w^0_h}_h v^1_h]^1) \ge 0. 
 \end{equation*}
\end{lemma}
\newcommand*{\wmaxi}{\overline{w}_{hi}}
\newcommand*{\wmini}{\underline{w}_{hi}}
\begin{proof}
 Proving the inequality works by going through the below cases. While each of the cases is easy to verify (see also \cite[Lem.\ 3.1]{BarrenecheaGPV23} to this end), it is essential to observe that they cover all possible combinations that $r_h^1$ and $v^1_h$ can create:
 \begin{itemize}
  \item $r_1, v_1 \ge b - \wmaxi$ or $r_1, v_1 \le  a - \wmini$:
  In this case $[P^{w^0_h}_h r^1_h]^1 = [P^{w^0_h}_h v^1_h]^1$.
  \item $a - \wmini \le r_1, v_1 \le b - \wmaxi$:
  In this case $[Q^{w^0_h}_h r^1_h]^1 = [Q^{w^0_h}_h v^1_h]^1$.
  \item $v_1 \le a - \wmini \le r_1$:
  $[P^{w^0_h}_h v^1_h]^1 = a - \wmini$,  $[P^{w^0_h}_h r^1_h]^1 \ge a - \wmini$, $[Q^{w^0_h}_h v^1_h]^1 \le 0$, $[Q^{w^0_h}_h r^1_h]^1 \ge 0$.
  \item $r_1 \le a - \wmini \le v_1$:
  $[P^{w^0_h}_h r^1_h]^1 = a - \wmini$,  $[P^{w^0_h}_h v^1_h]^1 \ge a - \wmini$, $[Q^{w^0_h}_h r^1_h]^1 \le 0$, $[Q^{w^0_h}_h v^1_h]^1 \ge 0$.
  \item $v_1 \le b - \wmaxi \le r_1$:
  $[P^{w^0_h}_h v^1_h]^1 = b - \wmaxi$,  $[P^{w^0_h}_h r^1_h]^1 \le b - \wmaxi$, $[Q^{w^0_h}_h v^1_h]^1 \ge 0$, $[Q^{w^0_h}_h r^1_h]^1 \le 0$.
  \item $r_1 \le b - \wmaxi \le v_1$:
  $[P^{w^0_h}_h r^1_h]^1 \le b - \wmaxi$,  $[P^{w^0_h}_h v^1_h]^1 = a - \wmini$, $[Q^{w^0_h}_h r^1_h]^1 \le 0$, $[Q^{w^0_h}_h v^1_h]^1 \ge 0$.
 \end{itemize}
\end{proof}

\begin{lemma}\label{LEM:T1_wellposed}
 For given $w^0_h \in V^0_h$ let us define the operator $\tilde T_1^{}\colon V^1_h \to [V^1_h]^\ast \simeq V^1_h$ via
 \begin{equation}\label{EQ:tilde_T}
  [\tilde T_1^{} r^1_h, v^1_h] = a_h([P^{w^0_h}_h r_h]^1, v^1_h) + s_h([Q^{w^0_h}_h r_h]^1, v^1_h) \qquad \text{ for all } v_h \in V^1_h.
 \end{equation}
 Then, $\tilde T_1^{}$ is continuous. Moreover, if
 $\alpha \ge C_\textup{equiv}$, $\tilde T_1^{}$ is 
strongly monotone in the sense that there is $C_M^{} > 0$ with
 \begin{equation}\label{EQ:strongly_monotone}
  [\tilde T_1^{} v^1_h - \tilde T_1^{} r^1_h, v^1_h - r^1_h] \ge C_M \| v^1_h - r^1_h \|^2_{a_h} \qquad \text{ for all } v^1_h, r^1_h \in V^1_h.
 \end{equation}
\end{lemma}
\begin{proof}
 The proof follows exactly the same lines as that of \cite[Thm.\ 3.2]{BarrenecheaGPV23} with \cite[Lem.\ 3.1: (25)]{BarrenecheaGPV23} replaced by Lemma \ref{LEM:generalize_Barranachea}.
\end{proof}

\subsection{Well-posedness of the iteration defined in \eqref{EQ:fixed_point}}
\begin{lemma}[Well-posedness of \eqref{EQ:T_1}]\label{EQ:wellposed_T1_real}
 Under the assumptions of Lemma \ref{LEM:T1_wellposed}, the equation \eqref{EQ:T_1} uniquely defines $w^1_h \in V^1_h$ for a given $w^0_h$. %If we have 
Moreover, if $\hat w^0_h, \bar w^0_h \in V^0_h$, and $\hat w^1_h, \bar w^1_h \in V^1_h$ are their images under $\tilde T_1$, then  the following Lipschitz continuity holds
 \begin{equation*}
  \| \hat w^1_h - \bar w^1_h \|_{a_h} \le \tfrac{1}{C_M} \| \hat w^0_h - \bar w^0_h \|_{a_h}.
 \end{equation*}
\end{lemma}
\begin{proof}
The existence and uniqueness of the solution of \eqref{EQ:T_1} are a direct consequence of the strong monotonicity and continuity of $\tilde T_1^{}$ (see, e.g., \cite[Thm.~10.49]{RenardyR04}). To prove the Lipschitz continuity, let $\hat w_h^0,\bar w_h^0\in V_h^0$ and $\hat w_h^1,\bar w_h^1\in V_h^1$ be their images under $\tilde T_1^{}$. Then, using \eqref{EQ:strongly_monotone} and the problem \eqref{EQ:T_1} we get to 
 \begin{equation*}
  C_M \| \hat w^1_h - \bar w^1_h \|_{a_h}^2 \le [\tilde T_1^{} \hat w^1_h, \hat w^1_h - \bar w^1_h] - [\tilde T_1^{} \bar w^1_h,\hat w^1_h - \bar w^1_h] = a_h(\bar w^0_h - \hat w^0_h, \hat w^1_h - \bar w^1_h),
 \end{equation*}
and the proof is finished using the continuity of $a_h(\cdot,\cdot)$.
\end{proof}

Once the proof that Step~1 is well-defined, we prove that the second step in the definition of $\widetilde T$ is also well-posed, thus proving that $\widetilde T$ is well-defined.

\begin{lemma}[Well-posedness of \eqref{EQ:T_2}]\label{LEM:t2_wellposed}
For any given $w^1_h \in V^1_h$, there exists a unique $\tilde w^0_h \in V^0_h$ solution of \eqref{EQ:T_2}. Moreover, if $\hat w^{1+}_h, \bar w^{1+}_h \in V^1_h$ and $\hat w^0_h, \bar w^0_h \in V^0_h$ are the solutions of \eqref{EQ:T_2} with  $\hat w^{1+}_h$ and $\bar w^{1+}_h$ as right-hand sides, respectively, the following Lipschitz continuity holds
 \begin{equation*}
 \| \hat w^0_h - \bar w^0_h \|_{a_h} \le \| \hat w^{1+}_h - \bar w^{1+}_h \|_{a_h}.
 \end{equation*}
\end{lemma}

\begin{proof} For $w_h^0,v_h^0\in V_h^0$ the bilinear form on the left-hand side of \eqref{EQ:T_2} reduces to
 \begin{equation}\label{EQ:lhs_T2}
  a_h(w^0_h, v^0_h) = \mu (w^0_h, v^0_h)_\mesh + \gamma \left\langle \frac{\epsilon+\mu h_\face^2}{h_\face^{\beta}} \jump{w^0_h},\jump{v^0_h} \right\rangle_\facets \qquad \text{ for all } w^0_h, v^0_h \in V^0_h.
 \end{equation}
This form is  continuous and $\| \jump{\cdot} \|_\facets^{}$-elliptic with ellipticity constant 
$\gamma h^{1-\beta}$. Thus, the problem in \eqref{EQ:T_2} is well-posed. To prove the Lipschitz continuity, we consider \eqref{EQ:T_2} with the two different right-hand sides, subtract the equations, and get to
 \begin{equation*}
     a_h^{}(\hat w^0_h- \bar w^0_h, v_h^0)= a_h^{}(\hat w^{1+}_h-  \bar w^{1+}_h, v_h^0),
 \end{equation*}
for every $v_h^0\in V_h^0$. Taking $v_h^0= \hat w^0_h- \bar w^0_h$ and using Cauchy-Schwarz's inequality on the right-hand side above leads to the result.
\end{proof}

\subsection{Proof of Theorem \ref{TH:existence}}\label{SEC:proof_existence}
We start proving the following a priori stability result for the solution of \eqref{EQ:T_2}.

\begin{lemma}\label{LEM:stab2}
Let $\beta\ge 2$, $w^1_h \in V^1_h$ be given, and let $\tilde w^0_h \in V^0_h$ be the
solution of \eqref{EQ:T_2} with $w^1_h \in V^1_h$ on the right-hand side. Then, there exists a
constant $C>0$, depending only on the mesh regularity, such that
 \begin{equation*}
 \| \jump{\tilde w^0_h} \|^2_\facets \le\, C\, \frac{h^{2\beta- 4}}{\gamma^2} \left[ C_P^2\mu^{-1}\| f \|^2_{0,\Omega} + (C_P^2\mu+\epsilon)\| w^{1+}_h \|^2_{0,\Omega}  \right].
 \end{equation*}
\end{lemma}
\begin{proof}
 Let $v^0_h = \tilde w^0_h$ in \eqref{EQ:T_2} and $\tilde \gamma = \tfrac{\gamma}{h^{\beta - 1}}$. Then, using \eqref{EQ:lhs_T2}, the ellipticity of $a_h(\cdot,\cdot)$ in $V_h^0$, \eqref{EQ:T_2}, Cauchy-Schwarz's inequality, the local trace-inverse inequality \eqref{Local:trace-inverse}, the broken Poincar\'e inequality from Lemma~\ref{LEM:poincare}, and a global inverse inequality,  we get to
 \begin{align*}
&\| \sqrt\mu  \tilde w^0_h \|^2_{0,\Omega} + \tilde \gamma \| \jump{ \tilde w^0_h } \|^2_\facets  = a_h(\tilde w^0_h, \tilde w^0_h) \\
  &=  (f, \tilde w^0_h)_\mesh^{} - a_h(w^{1+}_h, \tilde w^0_h) \\
  & = (f, \tilde w^0_h)_\mesh^{} - (\mu w^{1+}_h, \tilde w^0_h)_\mesh^{}- \langle \mean{\eps \nabla w_h^{1+}}, \jump{\tilde w^0_h} \rangle_\facets \\
  &\le C_P^{}\mu^{-\frac{1}{2}} h^{-1}\|f\|_{0,\Omega}^{}\| \jump{\tilde w^0_h} \|^{}_\facets
  + C_P^{}\sqrt{\mu}h^{-1}\|w^{1+}_h\|_{0,\Omega}^{}\| \jump{\tilde w^0_h }\|^{}_\facets
  + C\sqrt{\epsilon}\|\nabla w_h^{1+}\|_{0,\Omega}^{}\| \jump{ \tilde w^0_h} \|^{}_\facets\\
 &\le C\Big(C_P^{}\mu^{-\frac{1}{2}}h^{-1}\|f\|_{0,\Omega}^{}
  + C_P^{}\sqrt{\mu}h^{-1}\|w^{1+}_h\|_{0,\Omega}^{}
  + \sqrt{\epsilon}h^{-1}\|w_h^{1+}\|_{0,\Omega}^{}\Big)\| \jump{ \tilde w^0_h } \|^{}_\facets\,.
 \end{align*}
 The proof then follows from rearranging terms.
\end{proof}

\begin{lemma}\label{LEM:existence}
Let us assume that $\beta\ge 2$. Then, if $\gamma > 0$ is sufficiently large, $\tilde T \colon B_1(0) \to B_1(0)$, where $B_1(0) \subset V^0_h$ is the closed unit ball in $(V_h^0, \| \jump{\cdot} \|_\facets)$.
\end{lemma}
\begin{proof}
Using Lemma~\ref{LEM:stab1}, Lemma~\ref{LEM:stab2}, and the broken Poincar\'e inequality from Lemma~\ref{LEM:poincare} we get to 
\begin{align}\label{EQ:mean_stab} 
&\| \jump{ \tilde w^0_h } \|^2_\facets  \le \, C\,\frac{h^{2\beta- 4}}{\gamma^2} \left[ C_P^2\mu^{-1}\| f \|^2_{0,\Omega} + (C_P^2\mu+\epsilon)\| w^{1+}_h \|^2_{0,\Omega}  \right]\nonumber\\
& \le \, C\,\frac{h^{2\beta- 4}}{\gamma^2} \left[ C_P^2\mu^{-1}\| f \|^2_{0,\Omega} + (C_P^2\mu+\epsilon)\big(\max(|a|,|b|)^2 |\Omega| + \| w^0_h \|^2_{0,\Omega}\big)\right]\nonumber\\
& \le \, C\,\frac{h^{2\beta- 4}}{\gamma^2} \left[ C_P^2\mu^{-1}\| f \|^2_{0,\Omega} + (C_P^2\mu+\epsilon)\max(|a|,|b|)^2 |\Omega| +
C_P^2\frac{C_p^2\mu+\epsilon}{\epsilon+\mu h^2}\| \jump{ w^0_h } \|^2_{\facets}\right]\,.
\end{align}
The result is then proved by noticing that $\gamma$ can always be chosen large enough so that the last term in the right-hand side above can be hidden in the left-hand side, and the resulting right-hand side is smaller than $1$.
%to guarantee that the right-hand side above is smaller than $1$.
\end{proof}

The proof of Theorem~ \ref{TH:existence} appears as a corollary of the above results.

\begin{proof}[Proof of Theorem \ref{TH:existence}]
 Since the EG finite element space is finite-dimensional, the operator $\tilde T$ is continuous and maps the closed unit ball (a convex compact set) onto itself. Thus,  Brouwer's fixed-point Theorem (see, e.g. \cite[Thm.~10.41]{RenardyR04}) states that $\tilde T$ has at least one fixed point in $B_1(0)$, which is the desired result.
\end{proof}

\subsection{Worst-case criterion for positivity preservation}\label{SEC:positivity}
Lemma \ref{COR:positivity_preservation} implies that for each iterate $w_h \in V_h$ defined via \eqref{EQ:fixed_point}, we have $w^+_h \in \mathcal G$ if $\| w^0_h \|_{0,\infty,\Omega}^{} < \frac{b-a}2$. This section underlines that this criterion can always easily be satisfied if either $\beta$ is large enough, $\gamma$ is large enough, or $h$ is small enough. To this end, we start observing that using the inverse inequality, we get
\begin{equation*}
\| \tilde w^0_h \|_{0,\infty,\Omega}^{} \le 
C\,h^{-d/2} \| \tilde w^0_h \|_{0,\Omega},
\end{equation*}
which can be combined with Lemma \ref{LEM:poincare} and \eqref{EQ:mean_stab} to obtain
\begin{multline*}
 \| \tilde w^0_h \|_{0,\infty,\Omega}^{} \le\, C\,h^{-d/2} \| \tilde w^0_h \|_{0,\Omega} \le \, C\,h^{-d/2}\left\{\sum_{\face\in\facets}\frac{1}{h_\face^{}}\|\jump{w_h^0}\|^2_{0,\face}\right\}^{\frac{1}{2}} \le \, C \frac{h^{-d/2}}{\sqrt{\epsilon+\mu h^2}} \|w_h^0\|_{\facets}^{}\\
 \le \, C\,\frac{\min\{\epsilon^{-1/2},\frac{\mu^{-1/2}}h\}}{\gamma} h^{\beta-2-d/2}\left[ \tfrac{C_P^2}\mu \| f \|^2_{0,\Omega} + (C_P^2\mu+\epsilon)\big(\max(|a|,|b|)^2 |\Omega| +\|w_h^0\|_{0,\Omega}^2\big)\right]^{\frac{1}{2}}.
\end{multline*}
So, assuming that $\beta\ge 2+d/2$, or that $\gamma$ is large enough,  for $w_h^0\in V_h^0$ satisfying $ \| w^0_h \|^2_{0,\Omega}\le 1$, we conclude that $\| \tilde w^0_h \|_{0,\infty,\Omega}^{}\le (b-a)/2$, thus guaranteeing that the next iterate $\tilde w^+_h$ belongs to $\mathcal G$.

% --------------------------------------------------------------------------------------------------
\section{Convergence order estimates}\label{SEC:conbergence}
% --------------------------------------------------------------------------------------------------
% 
In this section, we prove optimal order error estimates for the present bound-preserving EG method. The order of convergence appears as a result of the over-penalization of the jump terms. More precisely, the stability bound \eqref{EQ:stab_result} proven in Theorem~\ref{TH:existence} is an error estimate in its own right, providing optimal convergence for $u_h^0$ if ${\beta\ge 4}$, since in such a case $u^0_h \to 0$ is sufficiently fast. Thus, the approximation properties of the current EG method stem from those of the continuous finite element subspace, while the piecewise constant enrichment is responsible for the local mass conservation. A similar interpretation of the best approximation properties can be used to argue for the convergence of DG methods, which converge optimally since the jump terms in the solutions converge to zero sufficiently quickly (the continuous solution does not have discontinuities). 
%However, over-penalization in standard DG/EG methods leads to ill-conditioned system matrices if no splitting approach as in \eqref{EQ:fixed_point} is employed.

We start by stating the following estimate for $u_h^0$. It is important to notice that this result is, in fact, a rewriting of Theorem~\ref{TH:existence}.
\begin{corollary}\label{Cor1}
    Under the assumptions of Theorem~\ref{TH:existence}, 
    Let $u_h^0$ be the piecewise constant part of a solution $u_h{}$ of \eqref{EQ:stab_prob}. Then, the following error estimate holds 
    \begin{equation}
        \| \jump{ u^0_h } \|^2_\facets 
\le \, C\,\frac{h^{{2\beta- 4}}}{\gamma^2} \left[ C_P^2\mu^{-1}\| f \|^2_{0,\Omega} + (C_P^2\mu+\epsilon)\max(|a|,|b|)^2 |\Omega| \right]\,.
\end{equation}
\end{corollary}
\begin{remark}
Thanks to the broken Poincar\'e inequality from Lemma~\ref{LEM:poincare}, we also have that $\|u_h^0\|_{0,\Omega}^{} \le C\, h^{\beta-{3/2}}$.
\end{remark}

We now state the main convergence result of this work.
\begin{theorem}\label{TH:convergence}
Let us assume that $u \in H^2(\Omega)$ solves \eqref{EQ:cont_prob}, $\beta\ge 4$, and that the assumptions of Theorem~\ref{TH:existence} hold. Then, the following error estimate holds
 \begin{equation*}
  \Big( \mu\| u - u^+_h \|_{0,\mathcal{T}_h}^2+\epsilon \|\nabla (u-u_h^+)\|_{0,\mathcal{T}_h}^2+\| \jump{ u_h^0 } \|_{\facets}^2\Big)^{1/2} \le C\,h(\sqrt{\mu} h+\sqrt{\epsilon})\, \Big( |u|_{2,\Omega} + \sqrt{\tilde{C}(f,a,b,\Omega)}\Big)\,,
 \end{equation*}
 where $\tilde{C}(f,a,b,\Omega)=\gamma^{-2}\left[ C_P^2\mu^{-1}\| f \|^2_{0,\Omega} + (C_P^2\mu+\epsilon)\max(|a|,|b|)^2 |\Omega| \right]$, and the constant $C$ does not depend on the mesh size, nor the physical coefficients of the problem. 
\end{theorem}

The first step in the proof of the error is the following result that states that the constrained part of the discrete solution $u_h^{+}$ satisfies a variational inequality.

\begin{lemma}\label{LEM:var_ineq}
 Let $u_h \in V_h$ solve \eqref{EQ:stab_prob} and define the closed convex set
 \begin{equation}\label{Vh+-specific}
  V^+_h = \{ v_h \in V_h\colon v_h = P^{u^0_h}_h w_h \text{ for some } w_h \in V_h \}.
 \end{equation}
 Then, the function $u^+_h$ satisfies
 \begin{equation*}
  a_h(u_h^+, v_h - u^+_h) \ge b(v_h - u_h^+) \qquad \forall v_h \in V^+_h.
 \end{equation*}
\end{lemma}
\begin{proof}
As in the proof of  \cite[Thm.\ 3.5]{BarrenecheaGPV23}, we can derive that
 \begin{equation*}
  a_h(u_h^+, v_h - u^+_h) + s(u^{1-}_h , v^{1}_h - u^{1+}_h) = b(v_h - u_h^+) \qquad \forall v_h \in V_h^+.
 \end{equation*}
 Using that $P^{u^0_h}_h v_h = v_h$ and $Q^{u^0_h}_h v_h = 0$ for $v_h \in V_h^+$ we exploit Lemma \ref{LEM:generalize_Barranachea} to deduce that
 \begin{equation*}
  s(u^{1-}_h, v^{1}_h - u^{1+}_h) = s([Q^{u^0_h}_h u_h]^1 - [Q^{u^0_h}_h v_h]^1, [P^{u^0_h}_h v_h]^1 - [P^{u^0_h}_h u_h]^1) \le 0 \qquad \forall v_h \in V^+_h,
 \end{equation*}
 which implies the result.
\end{proof}

To exploit the variational inequality proven above, we must choose an appropriate test function. To achieve this, we introduce the Lagrange interpolation operator $\interpol_h:C^0(\bar{\Omega})\to V_h^1$ defined in \cite[Ch.r~11]{EG21-I}, and define the test function
\begin{equation}\label{good-test}
\tilde{v}_h^{}:= P^{u^0_h}_h(\interpol_h^{} u).
\end{equation}
Then, the following result holds.
\begin{lemma}
The following error estimate holds
\begin{equation*}
%\|u-u_h^+\|_{a_h}^{}
\mu\|u-u_h^+\|_{0,\mathcal{T}_h}^2+\epsilon \|\nabla(u-u_h^+)\|_{0,\mathcal{T}_h}^2
\le C\Big( \epsilon\,\|\nabla(u-\tilde{v}_h^{})\|_{0,\mathcal{T}_h}^2+ \mu\|u-\tilde{v}_h^{}\|_{0,\mathcal{T}_h}^2+\| \jump{ u_h^{0} } \|_{\facets}^2\Big)\,,
\end{equation*}
where $C>0$ depends only on the mesh regularity.
\end{lemma}
\begin{proof}
    Thanks to the regularity assumption on $u$, the EG method used in this work is consistent, and thus
 \begin{equation*}
  a_h(u, v_h - u^+_h) = b(v_h - u^+_h)
\,,
\end{equation*}
for every $v_h^{}\in V_h$. So, for every $v_h^{}\in V_h^+$, where we remind that $V_h^+$ is given by \eqref{Vh+-specific}, the following inequality holds
\begin{equation*}
  a_h(u_h^+ - u, v_h - u^+_h) \ge 0\,.
 \end{equation*}
 Next, considering $\tilde{v}_h^{}$ defined in \eqref{good-test}, using the ellipticity of the bilinear form $a_h^{}(\cdot,\cdot)$, the Cauchy-Schwarz inequality, the local trace-inverse result \eqref{Local:trace-inverse}, the fact that $\tilde{v}_h^{}-u_h^+$ is continuous, and that $\jump{\tilde{v}_h^{}-u}= {\jump{ P^{u^0_h}_h(\interpol_h^{} u) } =}\jump{u_h^0}$
 %\arnote{Why does this hold? Also, I do not see the first inequality below.}
 we get
 \begin{align*}
     \|\tilde{v}_h^{}-u_h^+&\|_{a_h}^{2} \\
     &= a_h^{}(\tilde{v}_h^{}-u_h^+,\tilde{v}_h^{}-u_h^+)\\
     &= a_h^{}(\tilde{v}_h^{}-u, \tilde{v}_h^{}-u_h^+)+ \underbrace{a_h^{}(u-u_h^+,\tilde{v}_h^{}-u_h^+)}_{\le 0}\\
     &\le a_h^{}(\tilde{v}_h^{}-u, \tilde{v}_h^{}-u_h^+)\\
     &= \epsilon (\nabla (\tilde{v}_h^{}-u), \nabla (\tilde{v}_h^{}-u_h^+))_{\mathcal{T}_h}^{} + \mu(  \tilde{v}_h^{}-u, \tilde{v}_h^{}-u_h^+)_{\mathcal{T}_h}^{} -\langle \epsilon \mean{\nabla(\tilde{v}_h^{}-u_h^+)}, \jump{\tilde{v}_h^{}-u}\rangle_\facets^{}\\
     &\le C\Big( \epsilon \|\nabla (u-\tilde{v}_h^{})\|_{0,\mathcal{T}_h}^2+\mu\|u-\tilde{v}_h^{}\|_{0,\mathcal{T}_h}^2+ \|\jump{u_h^0}\|_{\facets}^2\Big)^{\frac{1}{2}}\|\tilde{v}_h^{}-u_h^+\|_{a_h}^{}\,,
 \end{align*}
 which, after using the triangle inequality, proves the result.
\end{proof}

To prove the error, it only remains to bound the difference {$u -\tilde{v}_h^{}=u-P_h^{u_h^0}(\interpol_h u)$}. We remark that $u_h^0$ has already been bounded in Corollary~\ref{Cor1}, so it only remains to bound the difference $u-{[P_h^{u_h^0}(\interpol_h u)]^1}$. The following result states that bound.

\begin{lemma}\label{LEM:plus_equiv}
Let $w_h { = w_h^1 + w_h^0} \in V_h$ be arbitrary. Then,  there exists a constant $C>0$ depending only on the mesh regularity such that
 \begin{align*}
  {\| w^1_h - [P_h^{w_h^0}(w_h)]^{1} \|_{0,\Omega} =}  \| w^1_h - w^{1+}_h \|_{0,\Omega} &\le C\, \| w_h^0 \|_{0,\Omega},  \\
  % \qquad \text{ and } \qquad \\
    {\| \nabla w^1_h - \nabla [P_h^{w_h^0}(w_h)]^{1} \|_{0,\Omega} =} \| \nabla w^1_h - \nabla w^{1+}_h \|_\mesh &\le C\, h^{-1} \| w_h^0 \|_{0,\Omega}.
 \end{align*}
\end{lemma}
\begin{proof}
 Let $T\in\mesh$. Then, using \eqref{Equiv:nodal-values} and the regularity of the mesh family, we get
 \begin{equation*}
 \| w^1_h - w^{1+}_h \|^2_{0,T} \le C\, h^{d} \max_{\bx_i \in \elem} (w^1_h(\bx_i) - w^{1+}_h(\bx_i))^2 \le C h^{d} \max_{\hat \elem \in \omega_\elem} \| w_h^0 \|^2_{0,\infty,\omega_T} \le C\, \| w_h^0 \|^2_{0,\omega_\elem}.
 \end{equation*}
 Adding over the elements in the mesh yields the first inequality. The second one follows using the inverse inequality.
\end{proof}

Gathering all these preliminary results, we can finally prove Theorem \ref{TH:convergence}.

\begin{proof}[Proof of Theorem \ref{TH:convergence}]
Classical approximation properties of $\interpol_h$ (see, e.g., \cite{EG21-I}) yield
 \begin{equation*}
     \sqrt{\mu}\| u - \interpol_h u \|_{0,\Omega}^{}+\sqrt{\epsilon}\|\nabla (u-\interpol_h u)\|_{0,\Omega}^{}\le C\, (\sqrt{\mu} h+\sqrt{\epsilon}) h |u|_{2,\Omega}\,.
 \end{equation*}
In addition, using Lemma~\ref{LEM:plus_equiv}, the approximation properties of $\interpol_h$ and Corollary~\ref{Cor1} we arrive at
\begin{align*}
   \| u - \tilde{v}_h^{}\|_{0,\Omega}^{}&\le \|u-\interpol_h u\|_{0,\Omega}^{}+ \|\interpol_h u-{[P^{u^0_h}_h(\interpol_h^{} u)]^1}\|_{0,\Omega}^{}+\|u_h^0\|_{0,\Omega}^{}\\
   &\le Ch^2 |u|_{2,\Omega}^{}+ C\|u_h^0\|_{0,\Omega}^{}\\
   &\le Ch^2 |u|_{2,\Omega}^{}+ C h^{\beta-2}\sqrt{\tilde{C}(f,a,b,\Omega)}\,,\\
   \| \nabla(u - \tilde{v}_h^{})\|_{0,\mathcal{T}_h}^{}&\le \| \nabla(u - \interpol_h^{}u)\|_{0,\mathcal{T}_h}^{}+\| \nabla( \interpol_h^{}u-{[P^{u^0_h}_h(\interpol_h^{} u)]^1})\|_{0,\mathcal{T}_h}^{}\\
   &\le Ch |u|_{2,\Omega}^{}+ Ch^{-1}\|u_h^0\|_{0,\Omega}^{}\\
   &\le Ch |u|_{2,\Omega}^{}+ C h^{\beta-3}\sqrt{\tilde{C}(f,a,b,\Omega)}\,.
\end{align*}
 Gathering the last inequalities and using that $\beta\ge 4$ proves the error estimate.
\end{proof}

% --------------------------------------------------------------------------------------------------
\section{Numerical examples}\label{SEC:numerics}
% --------------------------------------------------------------------------------------------------
This section presents various numerical examples to validate our theoretical findings and discuss implementation aspects. All computations were done using the finite element library Netgen/NGSovle, see \cite{Schoberl14}, \url{www.ngsolve.org}.

\subsection{An iterative algorithm}\label{SEC:algo}

To solve the nonlinear problem \eqref{EQ:stab_prob}, we follow a similar strategy as was presented in \cite{BarrenecheaGPV23}, i.e., we consider using a Richardson-like iterative approximation. However, since an iteration of the (linearized) fully coupled problem defined on $V_h$ might lead to very ill-conditioned system matrices, see Section \ref{SEC:condtioning} below, we define a nested iterative scheme that is motivated by the fixed point iteration \eqref{EQ:fixed_point} introduced in Section \ref{sec:postpreserv}. Each step in the iterative algorithm includes solving equations \eqref{EQ:T_1} and \eqref{EQ:T_2}. While \eqref{EQ:T_2} is linear and can be solved with classical tools (e.g., a direct solver), we introduce another iterative (sub) scheme to solve the non-linear problem \eqref{EQ:T_1}. 

To track the nested iterative algorithm, we use the following notation. The outer loop, resembling the iteration of the fixed point iteration \eqref{EQ:fixed_point}, is denoted by the index $m$. The inner loop, resembling the iteration of the linearized problem \eqref{EQ:T_1}, is denoted by the index $n$. Correspondingly, we denote use $u_{h}^{1}\big|^{n}_m$ to indicate the linear part of the solution of the linearized problem at the $n$-th inner iteration of the $m$-th outer iteration. Since $u_h\big|_m$ is fixed in the inner loop, we omit the index $m$ for the constants in the following. Now consider a tolerance $tol_m$ for the outer loop and a tolerance $tol_n$ for the inner loop. The algorithm is then defined as follows:

\begin{outerbox}%{Outer loop}
Given a starting guess $u_h^0\big|_0 \in V_h^0$
we solve for all $m = 0,1,...$ the problem for $u_h^0\big|_{m+1}$:\\

\textbf{Step 1}: Solve the following iteration to find $u_h^1\big|_{m+1}$:

\begin{innerbox}
    Given the starting guess $u_h^1\big|_m^0 = u_h^1\big|_m$ and a  damping parameter $\omega \in (0,1]$, for each $n = 0,...$ we find $u_{h}^{1}\big|^{n+1}_m \in V_h^1$ such that for all $v_h^1 \in V_h^1$ there holds
    \begin{align*}
        a_h(u_{h}^{1}\big|^{n+1}_m, v_h^1) = &a_h(u_h^{1}\big|_m^n,  v_h^1) \\
        &+ \omega \Big( (f,  v_h^1) - a_h([P_h^{u_h^0\big|_m} u_h^1\big|_m^{n}]^1,  v_h^1) 
        - s_h([Q_h^{u_h^0\big|_m} u_h^1|^{n}_m]^1,  v_h^1) \Big), 
    \end{align*}
    until $\| u_h^1\big|^{n+1}_m - u_h^1\big|^{n}_m\|_0 \le tol_n$. Let $N$ denote the number of iterations needed to reach the tolerance $tol_n$. We then set $u_h^1\big|_{m+1} = u_h^1\big|^N_{m}$.
\end{innerbox}

\textbf{Step 2}:
 To find $u_h^0\big|_{m+1}$, solve the problem:
\begin{align*}
    a_h(u_h^0\big|_{m+1}, v_h^0) = (f, v_h^0) - a_h(u_h^1\big|_{m+1}, v_h^0),  \quad \forall v_h^0 \in V_h^0.
\end{align*}
Terminate the (outer) iteration when $\| u_h^0\big|_{m+1} - u_h^0\big|_{m}\|_0 \le tol_m$. 
\end{outerbox}

We chose $\omega = 1$ in our numerical examples and tested different values for the stopping criteria. Surprisingly, we found that the tolerance of the inner loop is not crucial for the algorithm's convergence. We observed that the algorithm converges for a wide range of tolerances; see Section \ref{sec:smooth_solution} for more details. As a starting guess for the outer loop, we used $u_h^0\big|_0 = [u_h^{init}]^0$, i.e., the constant part of $u_h^{init}$ given as the solution of the standard EG problem without modification, i.e. 
\begin{align*}
    a_h(u_h^{init}, v_h) = (f, v_h), \quad \forall v_h \in V_h.
\end{align*}
We similarly used $u_h^1\big|_0 = [u_h^{init}]^1$ for the inner loop.

\subsection{The condition number of the linearized problems}\label{SEC:condtioning}

Let  $\mathbb{A}$ denote the finite element matrix associated to the bilinear form $a_h(\cdot, \cdot)$, that is, the matrix that needs to be inverted when solving a linearized problem of \eqref{EQ:stab_prob}. In the discussion that follows, we can neglect the stabilization term $s_h(\cdot,\cdot)$, as it only appears on the right-hand side of the inner iterations in Step 1 of the scheme presented in the last section. Further let $a_h^i(\cdot, \cdot)$ with $i \in \{0,1\}$ denote the restriction of the bilinear form $a_h$ on $V_h^i$, and let $\mathbb{A}^i$ denote the corresponding matrix.  In the following, we discuss the condition numbers we expect for each system matrix, considering varying choices of $\beta$. By construction, the matrix $\mathbb{A}^1$ is the standard finite element matrix of linear Lagrange finite elements and is independent of $\beta$. The condition $\kappa(\mathbb{A}^1)$ is then known to scale like $\mathcal{O}(h^{-2})$ which follows with standard techniques, see e.g. \cite[Ch.~28]{EG21-II}.

In the forthcoming discussion, we will write $a\lesssim b$ to denote that $a\le C\, b$, where $C>0$ is a constant independent of $h$ (but that might depend on the physical parameters $\mu$ and $\epsilon$), and $a\sim b$ if $a\lesssim b$ and $b\lesssim a$.
To discuss the condition number of $\mathbb{A}^0$ we introduce the bilinear forms 
\begin{align*}
    m(w_h^0, v_h^0) = \int_\Omega w_h^0 v_h^0 \dx, 
    \quad \text{and} \quad 
    j(w_h^0, v_h^0) = \langle \jump{w^0_h},\jump{v^0_h} \rangle,
\end{align*}
for $w_h^0, v_h^0 \in V_h^0$, and denote by %$M$ and $J$ 
$\mathbb{M}$ and $\mathbb{J}$ the corresponding system matrices. Then, assuming for simplicity $\mu = \gamma = 1$, we have that $\mathbb{A}^0 = \mathbb{M} + h^{-\beta} \mathbb{J}$ (with the obvious abuse of notation regarding the term containing the negative power of $h$). In the first step, we only consider the jump matrix $\mathbb{J}$. Let $u_h^0 \in V_h^0$ be arbitrary, and let $\uu_0$ denote the corresponding finite element coefficient vector.  

Using the broken Poincar\'e inequality from Lemma~\ref{LEM:poincare} and the inverse trace inequality \eqref{Local:trace-inverse} we obtain
\begin{align*}
    \uu_0^T \mathbb{M} \uu_0 = 
    \| u_h^0\|^2_{0,\Omega}
   \lesssim  \underbrace{h^{-1} \langle \jump{u^0_h},\jump{u^0_h} \rangle}_{h^{-1} \uu_0^T \mathbb{J} \uu_0} 
   \lesssim h^{-2} \| u_h^0\|^2_{0,\Omega}
   = h^{-2}  \uu_0^T \mathbb{M} \uu_0,
\end{align*}
that is, after multiplying by $h$, 
\begin{equation}
h\uu_0^T \mathbb{M} \uu_0 \lesssim  \uu_0^T \mathbb{J} \uu_0\lesssim h^{-1} \uu_0^T \mathbb{M} \uu_0\,,
\end{equation}
and using that $\kappa(\mathbb{M})\sim \mathcal{O}(1)$ (see, e.g., \cite[Prop.~28.6]{EG21-II}) we conclude that $\kappa(\mathbb{J}) \sim \mathcal{O}(h^{-2})$ (and thus is independent of the value of $\beta$). This further implies 
\begin{align*}
   % \uu_0^T M \uu_0 \le 
   (1 + h^{-\beta + 1} ) \uu_0^T \mathbb{M} \uu_0 
   \lesssim
   \underbrace{
   \uu_0^T( \mathbb{M} + h^{-\beta}  \mathbb{J}) \uu_0
   }_{\sim \uu_0^T \mathbb{A}^0 \uu_0}
   \lesssim
   (1 + h^{-(\beta + 1)}) \uu_0^T \mathbb{M} \uu_0,
\end{align*}
from which we also conclude 
$\kappa(\mathbb{A}^0) \sim \mathcal{O}\Big( (1 + h^{-(\beta + 1)})/(1 + h^{-\beta + 1} ) \Big)$
which is bounded by $C h^{-2}$
\textbf{independently} of the choice of $\beta$.

\pgfplotstableset{
EOCStyle/.style={
column name = {$eoc$}, column type={c},precision = 2, zerofill,  postproc cell content/.append style={
/pgfplots/table/@cell content/.add={(}{)}}
},
NUMStyle/.style={
sci, precision = 2, sci zerofill
}
}

\pgfplotstableset{
ConditionStyle/.style={
columns/h/.style={column name = $h$, precision = 4,column type = {r} },
    columns/k_a/.style={column name = $\kappa(\mathbb{A})^{-1}$,  NUMStyle},
    columns/k_a1/.style={column name = $\kappa(\mathbb{A}^1)^{-1}$, NUMStyle},
    columns/k_a0/.style={column name = $\kappa(\mathbb{A}^0)^{-1}$, NUMStyle},
    create on use/eta_k_a/.style={create col/expr={ln(\prevrow{k_a}/\thisrow{k_a})/ln(2)}},
    columns/eta_k_a/.style={EOCStyle},
    create on use/eta_k_a1/.style={create col/expr={ln(\prevrow{k_a1}/\thisrow{k_a1})/ln(2)}},
    columns/eta_k_a1/.style={EOCStyle},
    create on use/eta_k_a0/.style={create col/expr={ln(\prevrow{k_a0}/\thisrow{k_a0})/ln(2)}},
    columns/eta_k_a0/.style={EOCStyle},
    columns={h,k_a, eta_k_a, k_a1, eta_k_a1, k_a0, eta_k_a0}, 
 every last row/.style={after row=\bottomrule},
 clear infinite,
 empty cells with={\ensuremath{--}}
}
}

\pgfplotstableset{
ErrorplotStyle/.style={
columns/EL/.style={column name = $|\mesh|$,int detect, precision = 0,column type = {r},set thousands separator={}},
    columns/l2up/.style={column name = $\| u - u_h^{+}\|_{0, \Omega}$,  NUMStyle},
    columns/h1up_1/.style={column name = $\| \nabla (u - u_h^{+})\|_{0, \Omega}$, NUMStyle},
    columns/it/.style={column name = $\#its$, precision = 0},
    create on use/eoc_l2up/.style={create col/expr={ln(\prevrow{l2up}/\thisrow{l2up})/ln(2)}},
    columns/eoc_l2up/.style={EOCStyle},
    create on use/eoc_h1up_1/.style={create col/expr={ln(\prevrow{h1up_1}/\thisrow{h1up_1})/ln(2)}},
    columns/eoc_h1up_1/.style={EOCStyle},
    columns/eoc_h1up_0/.style={EOCStyle},
    columns={EL, l2up, eoc_l2up, h1up_1, eoc_h1up_1, it}, 
 every last row/.style={},
 clear infinite,
 empty cells with={\ensuremath{--}}
}
}

\paragraph{Numerical investigation:}
To validate the findings presented above, we computed the condition numbers for a structured triangulation of the domain $\Omega = (0,1)^2$ with various mesh sizes. The results, summarized in Table \ref{tab::condition}, demonstrate that $\kappa(\mathbb{A}), \kappa(\mathbb{A}^0)$ and $\kappa(\mathbb{A}^1)$ scale at the anticipated rates, consistent with the theoretical predictions discussed earlier. 

\begin{table}[]
    \centering
    \pgfplotstabletypeset[
    ConditionStyle,
    every head row/.style={before row={ \multicolumn{7}{c}{$\beta = 1$} \\}, after row=\midrule},
]{
h	k_a	k_a1	k_a0
0.5	0.0015234553584498364	0.020317687019874978	0.12785520512223322
0.25	0.0005137793904233167	0.005604012360017594	0.04123629655440426
0.125	0.00015499898824980585	0.0015926783490645307	0.011973101818079427
0.0625	4.31385512756838e-05	0.0004361437155001302	0.0032517711072445475
0.03125	1.142474293088148e-05	0.0001150303858377088	0.0008493687641849691
}
\pgfplotstabletypeset[
    ConditionStyle,
    every head row/.style={before row={ \multicolumn{7}{c}{$\beta = 2$} \\}, after row=\midrule},
]{
h	k_a	k_a1	k_a0
0.5	0.000622730545026675	0.020317687019874978	0.10631967314009769
0.25	0.00010633318019992562	0.005604012360017594	0.03529481432533098
0.125	1.6103196747576793e-05	0.0015926783490645307	0.010385339571664551
0.0625	2.2439022192891456e-06	0.0004361437155001302	0.002838002506999782
0.03125	2.972915982285604e-07	0.0001150303858377088	0.0007434754393226912
}
\pgfplotstabletypeset[
    ConditionStyle,
    every head row/.style={before row={ \multicolumn{7}{c}{$\beta = 4$} \\}, after row=\midrule},
]{
h	k_a	k_a1	k_a0
0.5	9.505724002822431e-05	0.020317687019874978	0.06688751036414015
0.25	4.148754036639504e-06	0.005604012360017594	0.023464310913774485
0.125	1.5824140125222578e-07	0.0015926783490645307	0.0070872013858186925
0.0625	5.524068133561655e-09	0.0004361437155001302	0.001959876641678632
0.03125	1.8306241074931717e-10	0.0001150303858377088	0.0005162693471985208
% ./tex/data/condition_4_data.out
}
    \caption{Condition numbers $\kappa(\mathbb{A}), \kappa(\mathbb{A}^0)$ and $\kappa(\mathbb{A}^1)$ obtained with a structured triangulation of the domain $\Omega = (0,1)^2$ with varying mesh sizes $h$.}
    \label{tab::condition}
\end{table}

\subsection{Smooth solution} \label{sec:smooth_solution}
Let $\Omega = (-1,1) \times (0,1)$ and the parameters be set to $\epsilon = 10^{-5}$ and $\mu = 1$. Further let $f$ be defined such that the exact solution of \eqref{EQ:cont_prob} is given by 
\begin{align*}
    u(x,y) = \sin( \pi (x+1)/2) \cdot \sin(\pi y).
\end{align*}
i.e., we have $a = 0$ and $b =1$. In the following, we investigate the convergence of the method on a sequence of (nested) unstructured triangulations $\mathcal{T}_h$. As suggested by the theory, we choose $\beta = 4$ (for the first two test cases) and the tolerance of the outer loop is $tol_m = 10^{-12}$. Further, following \cite{BarrenecheaGPV23}, we choose the damping parameter $\omega = 0.5$ and set $\alpha = 1$ in \eqref{EQ:stabilizer}.

\paragraph{Convergence of the method and choice of $tol_n$:}
In Table \ref{tab::convergence_smooth_ex}, we present the results for the convergence of the method for different tolerances of the inner loop and the penalty factor $\gamma = 10$. We observe that the method converges for a wide range of tolerances of the inner loop, and the number of (outer) iterations $it$ is only mildly influenced by the choice of $tol_{n}$. Further, we see that the estimated order of convergence $eoc$ of the errors is optimal, i.e., we observe a quadratic rate for the error $\| u - u_h^{+}\|_{0,\Omega}$, and a linear rate for $\| \nabla (u - u_h^{+})\|_{0,\Omega}$. Note, that the $H^1$-error only considers the linear part in $V_h^1$ but not the constant part in $V_h^0$, i.e. $\| \nabla (u - u_h^{+})\|_{0,\Omega} = \| \nabla (u - (u_h^{+})^1)\|_{0,\Omega}$. 

\begin{table}[]
    \centering
    \pgfplotstabletypeset[
    ErrorplotStyle,
    every head row/.style={before row={ \multicolumn{6}{c}{$tol_{n} = 10^{-3}$} \\}, after row=\midrule},
]{
EL	l2up	h1up_1	h1up_0	l2up_0	it
496	0.0053268353184241545	0.34409474456722167	9.070241571411266e-08	1.0707205156514822e-07	9
1984	0.0013025184662036147	0.1652503558560618	2.7570779752012828e-09	4.410883813075645e-09	8
7936	0.00032643545668526423	0.07809679219803074	9.849408847222093e-11	1.5750046148382107e-10	6
31744	8.368041236768351e-05	0.03763275371102328	5.098243405028739e-12	6.5947920408485055e-12	2
126976	2.1315599558076714e-05	0.01864114662150483	3.867097746367433e-13	2.0230820754225915e-13	1
% ./tex/data/convergence_smooth/smooth_error_beta4_gamma_10_innertol_0.001eps_1e-05_data.data
}
\pgfplotstabletypeset[
    ErrorplotStyle,
    every head row/.style={before row={\\ \multicolumn{6}{c}{$tol_{n} = 10^{-6}$} \\}, after row=\midrule},
]{
EL	l2up	h1up_1	h1up_0	l2up_0	it
496	0.005326835318611587	0.34409474427806486	9.070241694248987e-08	1.0707206542709531e-07	5
1984	0.0013025184662751861	0.16525035579967276	2.7570779965971162e-09	4.410884117670592e-09	5
7936	0.00032643545669429276	0.07809679219428534	9.849408870586608e-11	1.5750047063341557e-10	4
31744	8.368041236768416e-05	0.03763275371102318	5.0982434050286986e-12	6.594792040823184e-12	2
126976	2.1315599558076755e-05	0.018641146621504782	3.8670977463672877e-13	2.023082075164672e-13	1
% ./tex/data/convergence_smooth/smooth_error_beta4_gamma_10_innertol_1e-06eps_1e-05_data.data
}
\pgfplotstabletypeset[
    ErrorplotStyle,
    every head row/.style={before row={ \\\multicolumn{6}{c}{$tol_{n} = 10^{-9}$} \\}, after row=\midrule},
]{
EL	l2up	h1up_1	h1up_0	l2up_0	it
496	0.005326835316342375	0.3440947477848743	9.070239453817463e-08	1.0707190943963934e-07	3
1984	0.0013025184044383193	0.16525040452562312	2.7570589150275445e-09	4.41064495692102e-09	2
7936	0.0003264350351540063	0.07809696772390866	9.848320215336036e-11	1.5708639949745168e-10	2
31744	8.368162752546712e-05	0.03763239667466656	5.095026597780015e-12	6.099647644328529e-12	2
126976	2.1315730492178288e-05	0.018641126646743073	3.867112312466749e-13	2.0760337417627505e-13	1
% ./tex/data/convergence_smooth/smooth_error_beta4_gamma_10_innertol_1e-09eps_1e-05_data.data
}
    \caption{Error convergence and number of iterations for the example of Section \ref{sec:smooth_solution} with $tol_m = 10^{-12}, \gamma = 10,  \beta = 4$ and varying tolerances $tol_n$.}
    \label{tab::convergence_smooth_ex}
\end{table}

\paragraph{The choice of $\gamma$:}In Figure \ref{fig::iterations_smooth_ex} we present the number of iterations for different penalty factors $\gamma$ and a tolerance of the inner loop set to $tol_n = 10^{-9}$. In the left plot, the numbers are given for $\epsilon = 10^{-5}$ and on the right for $\epsilon = 1$. We observe that the number of iterations is barely influenced by the choice of $\gamma$, even for a small $\epsilon$. To motivate this behavior, we recall the stabilizing term in the bilinear form $a_h(\cdot, \cdot)$ given by $\langle \gamma \frac{\epsilon+\mu h_\face^2}{h_\face^{\beta}} \jump{w_h},\jump{v_h} \rangle$. We see that even in the case of small $\epsilon$ and a coarse mesh, since $\mu = 1$ and $\beta \ge 4$, we get a relatively big penalization of the jumps of the solution even for a moderate $\gamma$. This would not be the case (particularly for vanishing $\epsilon$) if we  use the penalty $\gamma \frac{\epsilon}{h_F^\beta}$ instead, see also Remark~\ref{rem::prop_EG}, point 5.

\pgfplotstableread[]
{
EL	l2up	h1up_1	h1up_0	l2up_0	it
496	0.005305843054010786	0.3443776611480764	0.00011761061101725302	0.0001106098514317654	9
1984	0.0012974470107324414	0.16533029683721984	2.8193932106892192e-05	2.1518615493459276e-05	9
7936	0.00032496475707489357	0.0781053585531581	8.23996573520508e-06	3.813666631694202e-06	10
31744	8.31066137845135e-05	0.03763308036145698	3.4743075091366835e-06	1.075577936351774e-06	11
126976	2.1013710062841414e-05	0.018641201404783158	2.117156493064559e-06	5.088962674515535e-07	13
%./data/convergence_smooth/smooth_error_beta1_gamma_10_innertol_1e-09eps_1e-05_data.data} 
}\gzeroXXboneXXefive %gamma^zero, beta = 1, eps = 1e-05

\pgfplotstableread[]
{
EL	l2up	h1up_1	h1up_0	l2up_0	it
496	0.0053248171052219465	0.3441256262726605	1.1088574789612546e-05	1.5019892697678394e-05	7
1984	0.001302277579058789	0.16525589310076846	1.3276038746353872e-06	2.3770179175238547e-06	7
7936	0.0003264011216964642	0.07809729820744601	1.8808263332262926e-07	3.101918443822782e-07	7
31744	8.367509581627882e-05	0.03763240852771218	3.9023696342345235e-08	4.201279432464422e-08	7
126976	2.131402417622253e-05	0.018641125592097733	1.1840496825216057e-08	5.261720759038095e-09	7
%./data/convergence_smooth/smooth_error_beta2_gamma_10_innertol_1e-09eps_1e-05_data.data} 
}\gzeroXXbtwoXXefive %gamma^zero, beta = 2, eps = 1e-05

\pgfplotstableread[]
{
EL	l2up	h1up_1	h1up_0	l2up_0	it
496	0.005326668960221333	0.3440971480586646	9.97172632102371e-07	1.2989940007443575e-06	4
1984	0.0013025079990931668	0.1652506332633215	6.013641955618179e-08	1.0753881655573759e-07	3
7936	0.0003264342924961023	0.07809697340619871	4.268723324237393e-09	7.654669129139836e-09	2
31744	8.368155502494644e-05	0.03763239682224326	4.416520485586122e-10	5.907844147208761e-10	2
126976	2.1315726110686563e-05	0.018641126028293948	6.699103389665889e-11	3.910660548702125e-11	2
%./data/convergence_smooth/smooth_error_beta3_gamma_10_innertol_1e-09eps_1e-05_data.data} 
} \gzeroXXbthreeXXefive %gamma^zero, beta = 3, eps = 1e-05

\pgfplotstableread[]{
EL	l2up	h1up_1	h1up_0	l2up_0	it
496	0.005326835316342375	0.3440947477848743	9.070239453817463e-08	1.0707190943963934e-07	3
1984	0.0013025184044383193	0.16525040452562312	2.7570589150275445e-09	4.41064495692102e-09	2
7936	0.0003264350351540063	0.07809696772390866	9.848320215336036e-11	1.5708639949745168e-10	2
31744	8.368162752546712e-05	0.03763239667466656	5.095026597780015e-12	6.099647644328529e-12	2
126976	2.1315730492178288e-05	0.018641126646743073	3.867112312466749e-13	2.0760337417627505e-13	1
% ./data/convergence_smooth/smooth_error_beta4_gamma_10_innertol_1e-09eps_1e-05_data.data} 
} \gzeroXXbfourXXefive %smoothzerofour %gamma^zero, beta = 4, eps = 1e-05

\pgfplotstableread[]
{
EL	l2up	h1up_1	h1up_0	l2up_0	it
496	0.005326614407366669	0.34409860265206166	1.2999816595482066e-06	2.0140489806844766e-06	5
1984	0.0013024633926458494	0.16525187509710781	3.073259460131034e-07	6.575329583102086e-07	5
7936	0.00032641963509484894	0.07809714613643542	8.530307630622613e-08	1.7225173608699547e-07	6
31744	8.367573739511142e-05	0.03763240885687244	3.5152711725095925e-08	4.1932045827349826e-08	7
126976	2.1312654576055548e-05	0.01864112597606922	2.1299188446441313e-08	8.565014759335054e-09	9
% ./data/convergence_smooth/smooth_error_beta1_gamma_1000_innertol_1e-09eps_1e-05_data.data} 
} \gthreeXXboneXXefive %gamma^three, beta = 1, eps = 1e-05

\pgfplotstableread[]
{
EL	l2up	h1up_1	h1up_0	l2up_0	it
496	0.005326830859899114	0.3440948575211394	1.1247570785015739e-07	1.6163204819553413e-07	3
1984	0.0013025164373616012	0.1652504553532146	1.344868455234822e-08	2.6853687005103267e-08	2
7936	0.0003264347113435074	0.07809697052546806	1.8919374465348097e-09	3.812530045128172e-09	2
31744	8.36815626852203e-05	0.037632396822524915	3.9077257073859035e-10	5.828599799447611e-10	2
126976	2.1315718533563136e-05	0.018641126031591106	1.1841931915571493e-10	7.467391436183456e-11	2
% ./data/convergence_smooth/smooth_error_beta2_gamma_1000_innertol_1e-09eps_1e-05_data.data} 
} \gthreeXXbtwoXXefive %gamma^three, beta = 2, eps = 1e-05

\pgfplotstableread[]
{
EL	l2up	h1up_1	h1up_0	l2up_0	it
496	0.005326849780272059	0.3440945526034218	9.9826031699179e-09	1.3069827279646687e-08	2
1984	0.0013025187811101246	0.16525039946547024	6.016763827786624e-10	1.0810992450712459e-09	2
7936	0.0003264350446368532	0.07809696766424262	4.269286852828116e-11	7.691188217197034e-11	2
31744	8.3681627617558e-05	0.03763239667463891	4.416603921067223e-12	5.932827994574468e-12	2
126976	2.1315730450111826e-05	0.018641126646764965	6.699108338331853e-13	3.9171557922348514e-13	1
% ./data/convergence_smooth/smooth_error_beta3_gamma_1000_innertol_1e-09eps_1e-05_data.data} 
}\gthreeXXbthreeXXefive %gamma^three, beta = 3, eps = 1e-05

\pgfplotstableread[]
{
    EL	l2up	h1up_1	h1up_0	l2up_0	it
496	0.005326851443059862	0.34409453422325903	9.0709375293755e-10	1.0712477652901434e-09	2
1984	0.0013025188843393788	0.16525039788330506	2.7571067446083893e-11	4.411573725100081e-11	2
7936	0.00032643505206510513	0.07809696760730542	9.848342245014887e-13	1.5710149681789225e-12	2
31744	8.368162397083477e-05	0.037632397647335426	5.095036444380154e-14	6.101367211868848e-14	1
126976	2.13157305461712e-05	0.018641126646718076	3.8671123296828095e-15	2.0760605001597222e-15	1
% ./tex/data/convergence_smooth/smooth_error_beta4_gamma_1000_innertol_1e-09eps_1e-05_data.data} 
} \gthreeXXbfourXXefive %\smooththreefour %gamma^three, beta = 4, eps = 1e-05

\pgfplotstableread[]{
EL	l2up	h1up_1	h1up_0	l2up_0	it
496	0.005326851607066976	0.34409453012779634	9.070945371674064e-13	1.0712535743873182e-12	1
1984	0.0013025189028002986	0.16525038708286838	2.7571114206029987e-14	4.411635587434758e-14	1
7936	0.0003264350641741133	0.0780969626171449	9.848373479553545e-16	1.5711347205655317e-15	1
31744	8.368162397907765e-05	0.03763239764732014	5.095036452797281e-17	6.101369828332865e-17	1
126976	2.1315730546716357e-05	0.018641126646717975	3.867112329856302e-18	2.0760607693864844e-18	1
% ./tex/data/convergence_smooth/smooth_error_beta4_gamma_1000000_innertol_1e-09eps_1e-05_data.data
} \gsixXXbfourXXefive %\smoothsixfour %gamma^six, beta = 4, eps = 1e-05

\pgfplotstableread[]
{
EL	l2up	h1up_1	h1up_0	l2up_0	it
496	0.005326851368404417	0.34409453565854253	1.302943704940873e-09	2.0339450162352084e-09	2
1984	0.001302518833969107	0.16525039886738385	3.0869912883507206e-10	6.79807292865062e-10	2
7936	0.00032643503681198047	0.07809696775628866	8.566384833545691e-11	1.929607482460396e-10	2
31744	8.368162242550677e-05	0.0376323966884866	3.521181320535499e-11	5.863577062976262e-11	2
126976	2.131573266400097e-05	0.018641126026040646	2.1303532929049236e-11	1.457616545869304e-11	2
% ./tex/data/convergence_smooth/smooth_error_beta1_gamma_1000000_innertol_1e-09eps_1e-05_data.data} 
} \gsixXXboneXXefive %gamma^six, beta = 1, eps = 1e-05

\pgfplotstableread[]
{
EL	l2up	h1up_1	h1up_0	l2up_0	it
496	0.005326851584947145	0.3440945328222438	1.1249353955327063e-10	1.6175810892454697e-10	2
1984	0.001302518886755417	0.16525039785743398	1.3450767252592217e-11	2.6889738587905978e-11	2
7936	0.00032643505189499235	0.0780969676090498	1.892088748478088e-12	3.821741643970442e-12	2
31744	8.368162391342365e-05	0.037632397647471394	3.907823943606156e-13	5.855061627880178e-13	1
126976	2.1315730529470512e-05	0.01864112664672681	1.184195309722162e-13	7.501021848680758e-14	1
% ./tex/data/convergence_smooth/smooth_error_beta2_gamma_1000000_innertol_1e-09eps_1e-05_data.data} 
}\gsixXXbtwoXXefive %gamma^six, beta = 2, eps = 1e-05

\pgfplotstableread[]
{
EL	l2up	h1up_1	h1up_0	l2up_0	it
496	0.005326851603821316	0.34409453259566375	9.982710468614053e-12	1.3070618998573991e-11	2
1984	0.00130251890269731	0.16525038708444745	6.016805944083034e-13	1.0811693162565102e-12	1
7936	0.0003264350641666911	0.07809696261720203	4.269309178317889e-14	7.692144776071556e-14	1
31744	8.368162397835466e-05	0.03763239764732168	4.416614385755408e-15	5.9345291380992136e-15	1
126976	2.1315730546619867e-05	0.018641126646717913	6.699108401367177e-16	3.9172610994735203e-16	1
% ./tex/data/convergence_smooth/smooth_error_beta3_gamma_1000000_innertol_1e-09eps_1e-05_data.data} 
} \gsixXXbthreeXXefive %gamma^six, beta = 3, eps = 1e-05

\pgfplotstableread[]{
EL	l2up	h1up_1	h1up_0	l2up_0	it
496	0.0126937306746898	0.30446750541519035	1.494045259450004e-05	8.434185087780678e-07	3
1984	0.003037188140250967	0.1506258431582411	9.084117778826116e-07	2.3176233680627144e-08	2
7936	0.0007438122037893364	0.07485974671974371	5.585371767233989e-08	6.723785665114485e-10	2
31744	0.00017905744941521147	0.03730705400864482	3.458253387417049e-09	1.98734284577646e-11	1
126976	4.399894694314133e-05	0.018621363605610487	2.1505902228202873e-10	6.048476158765675e-13	1
% ./tex/data/convergence_smooth/smooth_error_beta4_gamma_10_innertol_1e-09eps_1_data.data
} \gzeroXXbfourXXeone %\smoothzerofourone %gamma^zero, beta = 4, eps = 1

\pgfplotstableread[]{
EL	l2up	h1up_1	h1up_0	l2up_0	it
496	0.012687426607833644	0.3044690399109952	0.00016158187547219438	9.863471432385207e-06	3
1984	0.003036780917040394	0.1506258829485671	1.95276599455428e-05	5.365940573656357e-07	3
7936	0.0007437849390504711	0.07485974778389029	2.3942633390857297e-06	3.090972075889927e-08	2
31744	0.00017905640827082934	0.0373070540291233	2.9602141605813714e-07	1.815231667096876e-09	2
126976	4.39988887186114e-05	0.018621363606097317	3.678678380932234e-08	1.1008542846599882e-10	1
% ./tex/data/convergence_smooth/smooth_error_beta3_gamma_10_innertol_1e-09eps_1_data.data
} \gzeroXXbthreeXXeone %\smoothzerofourone %gamma^zero, beta = 3, eps = 1

\pgfplotstableread[]{
EL	l2up	h1up_1	h1up_0	l2up_0	it
496	0.012613855479333467	0.3044855405909922	0.0017813199450670938	0.00011681325908589382	10
1984	0.003027299812123496	0.15062670823464264	0.000428483968255375	1.2604241012880756e-05	6
7936	0.0007425603701787231	0.07485979103800926	0.00010485005545033115	1.4440803549177991e-06	4
31744	0.00017889882947250274	0.03730705591445376	2.5896772922916358e-05	1.6876155100990208e-07	3
126976	4.397519402768637e-05	0.01862136371320461	6.4323682403842315e-06	2.041230362246217e-08	2
% ./tex/data/convergence_smooth/smooth_error_beta2_gamma_10_innertol_1e-09eps_1_data.data
} \gzeroXXbtwoXXeone %\smoothzerofourone %gamma^zero, beta = 2, eps = 1

\pgfplotstableread[]{
EL	l2up	h1up_1	h1up_0	l2up_0	it
496	0.01175555509673319	0.30468967309558564	0.02006859507647562	0.0014298032952053314	21
1984	0.0028083164776925343	0.15064667720995997	0.009588732364804514	0.0003055402975824401	18
7936	0.0006862040164611112	0.07486189996285077	0.004681391616217421	6.952427297458192e-05	15
31744	0.00016445125055642538	0.03730724260534213	0.00230975072040356	1.6133992720751997e-05	13
126976	3.9557765980401734e-05	0.018621384873099774	0.0011467069823175005	3.832813400336701e-06	12
% ./tex/data/convergence_smooth/smooth_error_beta1_gamma_10_innertol_1e-09eps_1_data.data
} \gzeroXXboneXXeone %\smoothzerofourone %gamma^zero, beta = 1, eps = 1

\pgfplotstableread[]{
EL	l2up	h1up_1	h1up_0	l2up_0	it
496	0.01269431898252259	0.30446734867784325	1.4940436084374152e-07	8.434073041003379e-09	2
1984	0.0030372071706780634	0.15062584116810485	9.084117620071491e-09	2.3176209112224978e-10	2
7936	0.0007438122256535443	0.07485974669972724	5.58537177250883e-10	6.723781685612816e-12	1
31744	0.00017905746819060664	0.037307054008357696	3.4582533874712963e-11	1.9873428250502368e-13	1
126976	4.399894727480442e-05	0.01862136360560746	2.150590220654023e-12	6.048476118719538e-15	1
% ./tex/data/convergence_smooth/smooth_error_beta4_gamma_1000_innertol_1e-09eps_1_data.data
} \gthreeXXbfourXXeone %\smooththreefourone

\pgfplotstableread[]{
EL	l2up	h1up_1	h1up_0	l2up_0	it
496	0.01269432452130399	0.3044673471129565	1.4940435917805726e-10	8.434071744501803e-12	1
1984	0.0030372067161877248	0.1506258411610605	9.084117623273261e-12	2.3176205686832573e-13	1
7936	0.0007438122317050707	0.0748597466994675	5.585371772511307e-13	6.723781678620859e-15	1
31744	0.0001790574683793323	0.037307054008354684	3.458253387471853e-14	1.987342824829914e-16	1
126976	4.399894729120729e-05	0.018621363605607403	2.150590220632191e-15	6.048476119587548e-18	1
% ./tex/data/convergence_smooth/smooth_error_beta4_gamma_1000000_innertol_1e-09eps_1_data.data
} \gsixXXbfourXXeone %\smoothsixfourone

\begin{figure}
    \begin{center}
    \begin{tikzpicture}
    \begin{groupplot}[
        group style ={group size = 2 by 1,horizontal sep=50pt},
        width = 7cm,
        height = 6cm,
        grid = both,
        xlabel = {$|\mathcal{T}_h|$},
        % , % {$\| u - u_h^{+}\|_{L^2(\Omega)}$},
        % legend pos = east,
        legend style={at={(0.65,-0.35)},anchor=west, legend columns=6},
        scaled y ticks = false,
        y tick label style={/pgf/number format/fixed},
        % ymode = log,
        xmode = log,
        cycle list name=philipcolors
        ]
        \nextgroupplot[title = {$\#its, \epsilon = 10^{-5}$}]
        \addplot table[x=EL, y=it] {\gzeroXXbfourXXefive};
        \addlegendentry{$\gamma = 10$}
        \addplot table[x=EL, y=it] {\gthreeXXbfourXXefive};
        \addlegendentry{$\gamma = 10^3$}
        \addplot table[x=EL, y=it] {\gsixXXbfourXXefive};
        \addlegendentry{$\gamma = 10^6$}

        \nextgroupplot[title = {$\#its, \epsilon = 1$}]
        \addplot table[x=EL, y=it] {\gzeroXXbfourXXeone};
        % \addlegendentry{$\gamma = 10$}
        \addplot table[x=EL, y=it] {\gthreeXXbfourXXeone};
        % \addlegendentry{$\gamma = 10^3$}
        \addplot table[x=EL, y=it] {\gsixXXbfourXXeone};
        % \addlegendentry{$\gamma = 10^6$}

        % \nextgroupplot[title = {$\| {u_h^{+}}\|_{\mathcal{F}_h}$}, ymode=log]
        % \addplot table[x=EL, y=h1up_0] {\smoothzerofour};
        % \addplot table[x=EL, y=h1up_0] {\smooththreefour};
        % \addplot table[x=EL, y=h1up_0] {\smoothsixfour};

        % \addplot[black,dashed,very thick] table[x=EL, y expr={(\thisrow{EL})^(-5/2)}] {\smoothzerofour};
        % \addlegendentry{$\gamma = 10^6$}
    % \end{axis}
\end{groupplot}
\end{tikzpicture}
\end{center}
% \vspace{-0.5cm}
\caption{Number of iterations for the example of Section \ref{sec:smooth_solution} for different penalty factors $\gamma$ and with $tol_n = 10^{-9}$ and $tol_m = 10^{-12}$.}
\label{fig::iterations_smooth_ex}
\end{figure}

\begin{figure}
    \begin{center}
    \begin{tikzpicture}
    \begin{groupplot}[
        group style ={group size = 2 by 1,horizontal sep=50pt},
        width = 7cm,
        height = 6cm,
        grid = both,
        xlabel = {$|\mathcal{T}_h|$},
        % , % {$\| u - u_h^{+}\|_{L^2(\Omega)}$},
        % legend pos = east,
        legend style={at={(0.55,-0.35)},anchor=west, legend columns=6},
        scaled y ticks = false,
        y tick label style={/pgf/number format/fixed},
        % ymode = log,
        xmode = log,
        cycle list name=philipcolors
        ]
        \nextgroupplot[title = {$\#its, \gamma = 10, \epsilon = 10^{-5}$}]
        \addplot table[x=EL, y=it] {\gzeroXXboneXXefive};
        \addlegendentry{$\beta = 1$}
        \addplot table[x=EL, y=it] {\gzeroXXbtwoXXefive};
        \addlegendentry{$\beta = 2$}
        \addplot table[x=EL, y=it] {\gzeroXXbthreeXXefive};
        \addlegendentry{$\beta = 3$}
        \addplot table[x=EL, y=it] {\gzeroXXbfourXXefive};
        \addlegendentry{$\beta = 4$}

        %  \nextgroupplot[title = {$\#its, \gamma = 10^3, \epsilon = 10^{-5}$}]
        % \addplot table[x=EL, y=it] {\gthreeXXboneXXefive};
        % \addplot table[x=EL, y=it] {\gthreeXXbtwoXXefive};
        % \addplot table[x=EL, y=it] {\gthreeXXbthreeXXefive};
        % \addplot table[x=EL, y=it] {\gthreeXXbfourXXefive};

        \nextgroupplot[title = {$\#its, \gamma = 10, \epsilon = 1$}]
        \addplot table[x=EL, y=it] {\gzeroXXboneXXeone};
        \addplot table[x=EL, y=it] {\gzeroXXbtwoXXeone};
        \addplot table[x=EL, y=it] {\gzeroXXbthreeXXeone};
        \addplot table[x=EL, y=it] {\gzeroXXbfourXXeone};

\end{groupplot}
\end{tikzpicture}
\end{center}
% % \vspace{-0.5cm}
\caption{Number of iterations for the example of Section \ref{sec:smooth_solution} for different factors $\beta$ and with $tol_n = 10^{-9}$ and $tol_m = 10^{-12}$.}
\label{fig::iterations_smooth_ex_beta}
\end{figure}

\paragraph{The choice of $\beta$:} Although the theory requires that $\beta\ge 4$, we discuss in the following the convergence for different (smaller) choices $\beta = 2,3,4$. We choose $tol_n = 10^{-9}$ and set $\gamma = 10$.  Furthermore, in contrast to the previous example, we choose $\varepsilon=10^{-3}$, as our computations showed faster pre-asymptotic convergence of some errors for smaller choices. Consequently, we increased the diffusion coefficient to make the jumps more pronounced. In Table \ref{tab::beta_smooth_ex} we again present the error $\| \nabla (u - u_h^{+})\|_{L^2(\Omega)}$ but further present the values of the jump norm $\| \jump{u_h^{+}} \|_{\mathcal{F}_h} = \| \jump{(u_h^{+})^0}\|_{\mathcal{F}_h}$ and the $L^2$-norm $\| (u_h^{+})^0\|_{L^2(\Omega)}$. We observe that the error $\| \nabla (u - u_h^{+})\|_{L^2(\Omega)}$ is independent of the choice of $\beta$ and converges with the optimal order. Note that the same conclusions can be made for the error $\|  (u - u_h^{+})\|_{L^2(\Omega)}$, which is omitted here for brevity. The jump norm $\| \jump{u_h^{+}} \|_{\mathcal{F}_h}$ shows a faster convergence. This suggests that the convergence analysis presented in Corollary~\ref{Cor1} may be sharpened. More precisely, the numerical results indicate that \( \| \jump{u_h^{+}} \|_{\mathcal{F}_h} \) converges at a rate of order~$h^\beta$. As in the previous paragraph, this observation can be motivated by the fact that for $\beta = 1$, the penalty essentially reduces to a standard interior penalty method, for which linear convergence of order~$\mathcal{O}(h)$ is typically expected.

Since standard error estimates are often based on ellipticity with respect to $a_h$, as used in the first step of the proof of Lemma~\ref{LEM:stab2}, we can argue that for $\beta > 1$
\[
 \| \jump{u_h^{+}} \|_{\mathcal{F}_h}^2  = \| \jump{ u_h^0} \|_{\mathcal{F}_h}^2 \cdot \frac{h^{\beta-1}}{h^{\beta-1}} \lesssim h^{\beta-1} a_h(u_h^0, u_h^0) \lesssim h^{\beta-1} \cdot \mathcal{O}(h) = \mathcal{O}(h^\beta).
\]
In accordance to the discussions from the previous paragraph ("\textit{The choice of} $\gamma$"), Figure~\ref{fig::iterations_smooth_ex_beta} reports the number of outer iterations as a function of $\beta$ for a fixed $\gamma=10$. We make the following observations. First, in the right plot for $\epsilon=1$, all choices of $\beta$ provided a decreasing number of iterations with respect to the number of elements. Further, we see that the number of iterations on a fixed mesh also decreases as $\beta$ increases. This is explained by the stronger jump penalization for larger $\beta$, which drives the piecewise-constant component $u_h^0$ toward zero more rapidly (see Table \ref{tab::beta_smooth_ex}); hence the outer update (which corrects $u_h^0$) becomes smaller and fewer outer iterations are needed to meet the tolerance.\\
However, in the left plot for $\epsilon=10^{-5}$, we see that for $\beta = 1,2$ the iteration count does not decrease with respect to the number of elements, while for $\beta=3,4$ it does. Here we can clearly see the influence of $\gamma$ which has to chosen big enough (and obviously in dependence of $\epsilon$). Fortunately, as motivated by the theory, the choice $\beta = 4$ provides a robust number of iterations also for small $\epsilon$.

\pgfplotstableset{
ErrorplotStyleall/.style={
columns/EL/.style={column name = $|\mesh|$,int detect, precision = 0,column type = {r},set thousands separator={}},
    columns/l2up/.style={column name = {$\| u - u_h^{+}\|_{0,\Omega}$},  NUMStyle},
    columns/h1up_1/.style={column name = {$\| \nabla (u - u_h^{+})\|_{0, \Omega}$}, NUMStyle},
    columns/l2up_0/.style={column name = {$\| (u_h^{+})^0\|_{0, \Omega}$}, NUMStyle},
    columns/h1up_0/.style={column name = {$\| \jump{u_h^{+}}\|_{\mathcal{F}_h}$}, NUMStyle},
    create on use/eoc_l2up/.style={create col/expr={ln(\prevrow{l2up}/\thisrow{l2up})/ln(2)}},
    columns/eoc_l2up/.style={EOCStyle},
    create on use/eoc_h1up_1/.style={create col/expr={ln(\prevrow{h1up_1}/\thisrow{h1up_1})/ln(2)}},
    columns/eoc_h1up_1/.style={EOCStyle},
    create on use/eoc_h1up_0/.style={create col/expr={ln(\prevrow{h1up_0}/\thisrow{h1up_0})/ln(2)}},
    columns/eoc_h1up_0/.style={EOCStyle},
    create on use/eoc_l2up_0/.style={create col/expr={ln(\prevrow{l2up_0}/\thisrow{l2up_0})/ln(2)}},
    columns/eoc_l2up_0/.style={EOCStyle},
    % columns={EL, l2up, eoc_l2up, h1up_1, eoc_h1up_1, l2up_0, eoc_l2up_0, h1up_0, eoc_h1up_0}, 
    columns={EL, h1up_1, eoc_h1up_1, l2up_0, eoc_l2up_0, h1up_0, eoc_h1up_0}, 
 every last row/.style={},
 clear infinite,
 empty cells with={\ensuremath{--}}
}
}

\begin{table}[]
    \centering
    \pgfplotstabletypeset[
    ErrorplotStyleall,
    every head row/.style={before row={ \multicolumn{7}{c}{$\beta=1$} \\},
    after row=\midrule},
]{
EL	l2up	h1up_1	h1up_0	l2up_0	it
496	0.005655261627930733	0.3119065043125357	0.000348948038822953	0.0003159912266682878	12
1984	0.0014105428245257425	0.15139721913341117	0.00020108768656974856	9.844871740083095e-05	19
7936	0.00034404890746174427	0.07494953031422179	0.0001246224710360672	2.8518246452422256e-05	23
31744	8.476270120847223e-05	0.0373171041243312	6.931705412622905e-05	8.191017353823746e-06	28
126976	2.105619255877519e-05	0.0186224039021743	3.5761193980029094e-05	2.1002015862357376e-06	23
% ./tex/data/convergence_smooth/smooth_error_beta1_gamma_10_innertol_1e-09eps_0.001_data.data
}
\pgfplotstabletypeset[
    ErrorplotStyleall,
    every head row/.style={before row={\\ \multicolumn{7}{c}{$\beta = 2$} \\}, after row=\midrule},
]{
EL	l2up	h1up_1	h1up_0	l2up_0	it
496	0.0057181299669398535	0.311684296420849	3.226635484885639e-05	4.033549861434804e-05	8
1984	0.0014414299522002639	0.15135853602550653	9.098103820369704e-06	8.643209569302162e-06	11
7936	0.00035708190011123254	0.07494444992417253	2.794264941939642e-06	1.3554278496927352e-06	11
31744	8.886799373985763e-05	0.0373166147803779	7.771395957497758e-07	1.831764907835323e-07	7
126976	2.2153137616195464e-05	0.01862238774619563	2.005862075402643e-07	2.0284469486641067e-08	2
% ./tex/data/convergence_smooth/smooth_error_beta2_gamma_10_innertol_1e-09eps_0.001_data.data
}
\pgfplotstabletypeset[
    ErrorplotStyleall,
    every head row/.style={before row={\\ \multicolumn{7}{c}{$\beta = 3$} \\}, after row=\midrule},
]{
EL	l2up	h1up_1	h1up_0	l2up_0	it
496	0.005724873050655195	0.3116611019776811	2.8784471716850823e-06	3.471729177173912e-06	3
1984	0.0014430802520916418	0.15135624380549242	4.101345403302213e-07	3.8165947848015005e-07	3
7936	0.00035742065269994954	0.07494429735296455	6.345396612456604e-08	3.008375157929989e-08	2
31744	8.892069930983287e-05	0.03731660807949101	8.865673292698793e-09	2.007459390639887e-09	2
126976	2.2159999139969726e-05	0.018622387705573902	1.1465192520963946e-09	1.1080135676709274e-10	2
% ./tex/data/convergence_smooth/smooth_error_beta3_gamma_10_innertol_1e-09eps_0.001_data.data
}
\pgfplotstabletypeset[
    ErrorplotStyleall,
    every head row/.style={before row={ \\\multicolumn{7}{c}{$\beta = 4$} \\}, after row=\midrule},
]{
EL	l2up	h1up_1	h1up_0	l2up_0	it
496	0.005725471333065051	0.31165913769380266	2.5979717112130723e-07	2.870917969323652e-07	3
1984	0.001443154375819953	0.15135614316890336	1.880554916484644e-08	1.5903814563340148e-08	2
7936	0.0003574283532759761	0.07494429351412576	1.4704137413474487e-09	6.361545487696744e-10	2
31744	8.892129357559207e-05	0.037316607996808006	1.0333178135981105e-10	2.1530009506662544e-11	2
126976	2.216003354330325e-05	0.01862238771248032	6.698341573562167e-12	6.058654974314214e-13	1
% ./tex/data/convergence_smooth/smooth_error_beta4_gamma_10_innertol_1e-09eps_0.001_data.data
}
\caption{Error convergence for the example of Section \ref{sec:smooth_solution} with the diffusion coefficient $\varepsilon = 10^{-3}$, penalty $\gamma = 10$ and tolerances $tol_n = 10^{-9}$ and $tol_m = 10^{-12}$ and varying $\beta = 1,2,3,4$.}
\label{tab::beta_smooth_ex}
\end{table}

\subsection{Interior layer example}\label{sec:interior_layer}

We consider the domain $\Omega = (0,1)^2$ and choose the parameters $\varepsilon = 10^{-7}$ and $\mu = 1$. This time, the right-hand side is given by 
\begin{align*}
    f = 
    \begin{cases}
        0 & \text{ in } [\tfrac{1}{4}, \tfrac{3}{4}]^2, \\
        1 & \text{ else},
    \end{cases}
\end{align*}
while we impose homogeneous Dirichlet boundary conditions on the whole boundary $\partial \Omega$. The exact solution of this problem is not known analytically. However, we know that the solution satisfies $a = 0 \le u \le 1 = b$ in $\Omega$.
Due to the discontinuous right-hand side and the small parameter $\varepsilon$ we expect that the solution will exhibit an interior layer. 

To investigate this in more detail, Figure~\ref{fig::interior_layer_3d} shows the solution of the standard EG method and the solution $u_h^+$ of our proposed method for different penalty parameters. Herein, we show results obtained using $\gamma = 10$, but larger values of $\gamma$ give very similar results. For the standard method, we directly solved the problem (i.e., without using an iterative scheme) and employed $\beta = 1$, $\gamma = 10$, and $\alpha = 0$ (no stabilization by \eqref{EQ:stabilizer}). The computations were made on a mesh with $|\mesh| = 242$ and the same tolerances and damping parameters as in the previous section. We observe that the solution of the standard EG method exhibits a highly oscillatory behavior in the interior layer and fails to preserve the limits. In fact, the values of the solution on the left range from approximately $-0.2$ to $1.2$, as can be seen in the left panel of Figure~\ref{fig::interior_layer_3d}. In contrast, the solution of our proposed method is bound-preserving. We want to emphasize that, although the solution appears to be approximated solely by the linear Lagrange Finite element - similarly to \cite{BarrenecheaGPV23} - our method is locally conservative since the piecewise constant part of $u_h^+$ is very small thanks to the over-penalization. 

One final point is that, since the constant part of $u_h^+$ is very small and the linear part strictly satisfying the homogeneous Dirichlet boundary conditions, the overall solution $u_h^+$ is very close to the boundary datum on the boundary. In contrast to that, the piecewise constants in the standard EG method (left panel) are not small conveying the impression that the standard EG method does not respect the Dirichlet boundary condition, as can be expected since it is only enforced via Nitsche's method (guaranteeing that the boundary condition holds strictly in the limit \( h \to 0 \)).

\begin{figure}
    \begin{center}
    \includegraphics[width=0.3\textwidth, clip = true, trim = 0 1cm 0 0]{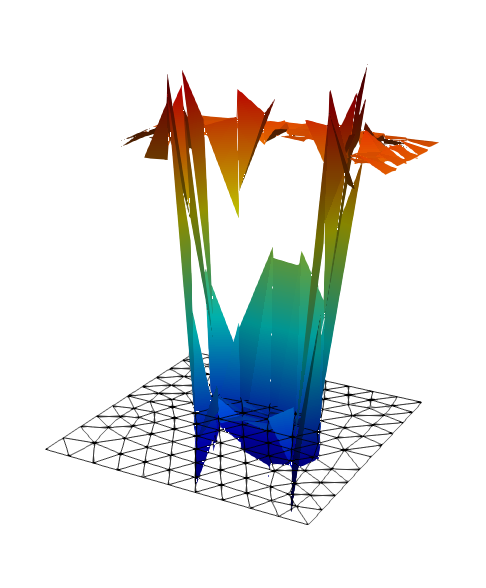}
    \includegraphics[width=0.3\textwidth, clip = true, trim = 0 1cm 0 0]{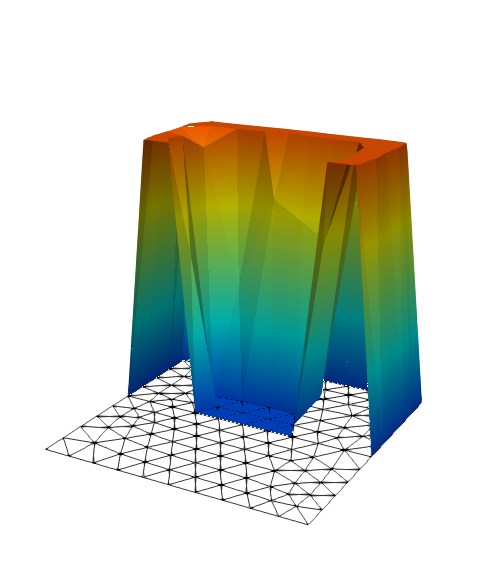}
    \begin{tikzpicture}
        \pgfplotscolorbardrawstandalone[colormap/jet,colorbar sampled,point meta min=-0.2,point meta max=1.2, 
            colorbar style={samples=22, width=0.2cm, height = 4cm, ytick={-0.2,0,0.2, 0.4, 0.6, 0.8, 1.0, 1.2}, 
            scaled y ticks=false, 
            yticklabel style={style={font=\footnotesize}, /pgf/number format/fixed,  /pgf/number format/precision=1}}];
    \end{tikzpicture}
\end{center}
\caption{Solutions of the interior layer example of Section \ref{sec:interior_layer} for the standard EG method (left) and the proposed method with $\gamma = 10$ (right). For better visualization, the solution on the partial domain $[0,1] \times [0.5,1]$ is shown here.}
 \label{fig::interior_layer_3d}   
\end{figure}

% --------------------------------------------------------------------------------------------------
\section{Conclusion}
% --------------------------------------------------------------------------------------------------
% 
We have proposed a bound-preserving EG method whose solution is locally (and globally) conservative. A fundamental tool to achieve this is the way we have built the limiting process, by leaving the piecewise constant part free, and limiting the piecewise linear part in such a way that the sum respects the bounds given by the continuous problem. As a result, the stabilization needed to compensate for the ``unconstrained'' part of the solution does not depend on the piecewise constant part of it, which implies conservation. In addition, it is important to insist on the fact that, although for the analysis a variational inequality was used, the method itself is not equivalent to a variational inequality (unlike \cite{BarrenecheaGPV23,AmiriBP23,BPT25}). The method is proven to approximate smooth solutions with optimal convergence rates if the analytical solution respects pre-defined upper and lower bounds. If the solution is not smooth, our approach still respects the upper and lower bounds, is mass conservative, and convergent, but the significance of the penalty parameter $\gamma$ increases: a larger $\gamma$ suppresses oscillations stronger than a smaller $\gamma$. However, the nonlinear problem remains solvable independent of our tuning parameters $\gamma$ and $\beta$ (which we fix to $4$). In addition, the use of the splitting algorithm presented in Section \ref{SEC:algo}  allows us to completely bypass the ill-conditioning of the linear systems that would arise were \eqref{EQ:stab_prob} to be solved in a monolithic way.

Notably, our analysis requires \( \beta \ge 2 \) for the existence of discrete solutions and \( \beta \ge 4 \) for optimal convergence rates. However, numerical experiments indicate that the method performs well even for \( \beta = 1 \), which is commonly used in interior penalty methods. This suggests that the theoretical requirements on \( \beta \) may be too restrictive, and further investigation into this aspect could be a valuable direction for future research. In addition, the following problems are also open at the moment, and investigation on them is ongoing:
\begin{itemize}
 \item An adaptive strategy to select $\beta$ and $\gamma$ in the spirit of Section \ref{SEC:positivity}.
 \item The consideration of non-linear and hyperbolic equations.
 \item The open problem whether the super-convergence of the $L^2$ error can be explained by a more involved proof strategy that allows for mimicking the Aubin--Nitsche trick.
 \item The proposal of, possibly a posteriori, strategies to fix the stopping criteria for the nested fixed-point iterations in an adaptive way, with the aim of reducing computational complexity.
\end{itemize}

% --------------------------------------------------------------------------------------------------

\bibliographystyle{siamplain}
\bibliography{references,references_gpt}

\end{document}